\documentclass[a4paper, 11pt]{amsart}
\usepackage{amscd,amsthm,amsfonts,latexsym,amssymb}
\usepackage[left=2.5cm, right=2.5cm, top=2cm]{geometry}

\usepackage{amsfonts}
\usepackage{amsthm}
\usepackage{graphicx}
\usepackage{float}

\usepackage{tikz}
\usetikzlibrary{intersections, calc}

\theoremstyle{plain}
\newtheorem{theorem}{Theorem}
\newtheorem{proposition}[theorem]{Proposition}
\newtheorem{lemma}[theorem]{Lemma}
\newtheorem{corollary}[theorem]{Corollary}

\theoremstyle{definition}

\newtheorem{example}[theorem]{Example}
\newtheorem{remark}[theorem]{Remark}%
\newcommand{\RR}{{\mathbb R}}  
\newcommand{\CC}{{\mathbb C}}  

\newcommand{\ZZ}{{\mathbb Z}}  
\newcommand{\NN}{{\mathbb N}}  

\newcommand{\ii}{{\rm i}}  

\newcommand{\ra}{\rightarrow}

\newcommand{\ds}{\displaystyle}

\renewcommand{\phi}{\varphi}

\newcommand{\al}{\alpha}
\newcommand{\be}{\beta}
\newcommand{\ga}{\gamma}

\newcommand{\ta}{\theta}

\newcommand{\wt}{\widetilde} 
\newcommand{\ov}{\overline} 

\renewcommand{\Im}{{\rm Im}\,} 

\newcommand{\m}{\medskip}
\newcommand{\n}{\noindent}

\newcommand{\verteq}{\rotatebox{90}{$\,=$}}

\title{Quadratic cyclic sequences}
\author{Paul Baird, Ali Fardoun and Zeina Ghazo Hanna}
\thanks{The authors thank Olivier Rahavandrainy and Benoit Saussol for helpful comments in respect of this work.}

\address{Laboratoire de Math\'ematiques de Bretagne Atlantique UMR 6205 \\
Universit\'e de Brest \\
6 av.\ Victor Le Gorgeu -- CS 93837 \\
29238 Brest Cedex, France}

\email{Paul.Baird@univ-brest.fr, Ali.Fardoun@univ-brest.fr, zeina.ghazohanna@univ-brest.fr}

\begin{document}

\begin{abstract} 
We explore relations between cyclic sequences determined by a quadratic difference relation, cyclotomic polynomials, Eulerian digraphs and walks in the plane. These walks correspond to closed paths for which at each step one must turn either left or right through a fixed angle. In the case when this angle is $2 \pi /n$, then non-symmetric phenomena occurs for $n\geq 12$. Examples arise from algebraic numbers of modulus one which are not $n$'th roots of unity.   
\end{abstract}



\maketitle

\maketitle

\section{Introduction}  \noindent For a given integer $N\geq 2$, define a \emph{quadratic cyclic sequence (QCS) of order $N$} to be a function $\phi : \ZZ / N\ZZ \ra \CC$ satisfying the quadratic difference relation 
\begin{equation}\label{one}
\frac{\ga}{2}\big( 2\phi (j) - \phi (j-1)-\phi (j+1)\big)^2 = \big(\phi (j) - \phi (j-1)\big)^2 + \big(\phi (j) - \phi (j+1)\big)^2 \qquad \forall j \in \ZZ / N \ZZ
\end{equation}
for some real number $\ga$ where $j\pm 1$ are calculated modulo $N$. If we define the \emph{increment} $u_j = \phi (j+1) - \phi (j)$, then the above equation reads
\begin{equation} \label{gamma-u}
\frac{\ga}{2} (u_j-u_{j-1})^2 - u_j{}^2 - u_{j-1}{}^2=0,
\end{equation}
which affirms the vanishing of a linear combination of the elementary symmetric quadratic polynomials $u_j{}^2+u_{j-1}{}^2$ and $u_ju_{j-1}$ in the two variables $u_j$ and $u_{j-1}$.  Such equations arise in studies of projections to the plane of regular polytopes \cite{EP} and invariant spacial frameworks \cite{Ba}.  

Equation \eqref{one} is invariant (for fixed $\ga$) under affine linear transformations and conjugation in the complex plane:
\begin{equation} \label{two} 
\phi \mapsto  a\phi + b  \quad \text{and} \quad \phi \mapsto \ov{\phi}\,, \qquad \forall \, a,b \in \CC \  {\rm with} \  a  \neq 0\,.
\end{equation}
Since there exists such a transformation mapping any pair of distinct points to any other pair of distinct points, we can normalize a QCS so that two distinct terms take on two distinct specified values. 
The equation is also invariant under cyclic permutations and order reversal of $(\phi (0), \phi (1), \ldots , \phi (N-1))$.

By the Cauchy-Schwarz inequality, for any set $\{ a_1, \ldots , a_{k}\}$ of complex numbers, one has
$$
\bigg| \sum_{j=1}^{k}a_j\bigg|^2 \leq k\sum_{j=1}^{k}|a_j|^2 
$$
with equality if and only if $a_1=a_2= \cdots = a_{k}$.  It follows from \eqref{gamma-u} that for a given $j \in \{ 0, 1, \ldots , N-1\}$, if $u_{j}, u_{j-1}$ are both real, we have
$$
u_j{}^2 + u_{j-1}{}^2 = \frac{\ga}{2} (u_j-u_{j-1})^2 \leq \ga (u_j{}^2 + u_{j-1}{}^2)\,.
$$
In particular, for there to exist a non-constant \emph{real} solution to \eqref{one}, necessarily $\ga \geq 1$. Equally, if $\ga < 1$, then for any three successive terms, at least one must be complex.

An example of an integer QCS of order $10$ is given by
\begin{equation} \label{ex1}
\left( 0, 9, 3 , 12, 6, 10, 4, 8, 2, 6\right)
\end{equation}
This satisfies \eqref{one} with $\ga = 26/25$.  

On applying the normalization \eqref{two}, a QCS is defined up to addition and multiplication by a constant.  Given a rational sequence, we may therefore multiply through by the smallest common multiple of the denominators, subtract the value of the first term and finally divide by any common factor in the subsequent numerators, to obtain an integer sequence $\phi : \ZZ / N \ZZ \ra \NN$ with $\phi (0) = 0$ and with no common factor.  Even then it may not be unique, for example the following is another sequence of order $10$ with $\ga = 26/25$:
$$
\left( 0,9,3,7,1,10,4,8,2,6\right)\,.
$$

A complex cyclic sequence can be respresented by a walk in the plane.  An example with $\ga = 2/3$ and $n = 6$ is given by
\begin{equation} \label{ex2}
\left(0, 1,\tfrac{1}{2} - \tfrac{\sqrt{3}}{2}\ii, 0,1, \tfrac{1}{2} + \tfrac{\sqrt{3}}{2}\ii\right) 
\end{equation}
with corresponding walk illustrated in Fig.\ref{fig:triangles}, where we label the vertices in sequential order. 

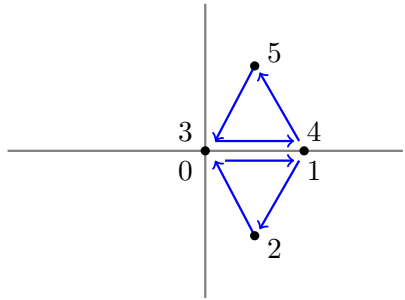
\begin{figure}[h] \label{fig:triangles}

\begin{tikzpicture}[line width=0.3mm,black,scale=1.3]
\draw [gray] (-2,0) -- (2,0);
\draw [gray] (0,1.5) -- (0,-1.5);
\draw [line width=0.3mm, blue,->] (0.2,-0.1) -- (0.9,-0.1);
\draw [line width=0.3mm, blue,->] (0.95,-0.1) -- (0.55,-0.8);
\draw [line width=0.3mm, blue,->] (0.5,-0.866) -- (0.1,-0.1);
\draw [line width=0.3mm, blue,->] (0.1,0.1) -- (0.9,0.1);
\draw [line width=0.3mm, blue,->] (0.95,0.1) -- (0.55,0.8);
\draw [line width=0.3mm, blue,->] (0.5,0.866) -- (0.1,0.1);
\filldraw [black] (0,0) circle (1pt);
\filldraw [black] (1,0) circle (1pt);
\filldraw [black] (0.5,0.866) circle (1pt);
\filldraw [black] (0.5,-0.866) circle (1pt);
\node at (-0.2,-0.2) {$0$};
\node at (1.1,-0.2) {$1$};
\node at (0.7,-1) {$2$};
\node at (-0.2,0.2) {$3$};
\node at (1.1,0.2) {$4$};
\node at (0.7,1) {$5$};
\end{tikzpicture}  

\caption{Complex cyclic sequence}
\end{figure}

We first examine real QCS, showing how they arise from polynomials with positive integer coefficients.  A complete characterization is given in Theorem \ref{thm:cyclic}.  This theorem in particular shows how a given polynomial can give rise to different sequences coming from \emph{legitimate} orderings of a corresponding set of increments. The two real sequences given above arise from different legitimate orderings. It turns out that we can capture the legitimate orderings by Eulerian walks in a corresponding digraph, a model we discuss later on in \S \ref{sec:euler}. 

Complex sequences which arise from polynomials will be called \emph{algebraic}. In this case, a \emph{legitimate} polynomial $p(x)$ determines a closed polygonal walk in the plane with exterior angle either $+\ta$ or $-\ta$ for some fixed angle $\ta$ (the turning angle). Complex algebraic QCS exist with turning angle not a rational multiple of $2\pi$ (Fig.\,3).   In the case when $\ta = 2\pi /n$, necessarily the $n$'th cyclotomic polynomial $\Phi_n(x)$ must divide $p(x)$. The problem of which polynomials can arise turns out to be challenging when $n$ becomes large and bears on the following geometric question. 

One wishes to construct a polygonal path starting at the origin, with directed edges taken from the edges of a closed regular polygon with exterior angle $2\pi /n$.  One may use edges as often as one likes, but at each step, the turning angle must be either $+2\pi /n$ or $- 2\pi /n$ (one turns left or right through an angle $2\pi /n$). In the case when the polygonal path be closed and $n$ is even, must each edge occur with its parallel counterpart oriented in the opposite direction? This is so for $n = 2,4,\ldots , 10$, but as we show in \S \ref{sec:complex}, this fails to hold in general for $n \geq 12$. Another issue is whether, to complete a circuit all edges are required (now for $n$ even or odd). Again, this is so for $n\leq 11$ but not in general when $n\geq 12$ (Proposition \ref{prop:symmetry}, Fig.\,2).

It is possible to combine sequences with common $\ga$ to obtain new sequences. In \S\ref{sec:three}, we see how this process of concatenation is reflected in the defining polynomials. 

In \S\ref{sec:random}, we study the unicity of the edges used to construct a walk with turning angle $2\pi /n$ of given length and end point. For walks on square, triangular and hexagonal lattices, unicity occurs when $n=4$ (Theorem \ref{thm:poly-n4}), but not in general when $n=6$ (Example \ref{ex:different-paths}).

\section{Construction of real quadratic cyclic sequences} \label{sec:rqcs}

\n  In what follows, we refer to \emph{normalization} as the freedom \eqref{two}.  By a \emph{real} QCS, we mean one in which every term is real under some normalization.  From the Introduction, for such sequences, we must have $\ga \geq 1$.  
Below, we will see that necessarily $ \ga \leq 2$.  First we deal with the case when $\ga =2$.

Let $(x_0,x_1,x_2,\ldots ,x_{N-1})$ be a QCS (not necessarily real) which solves (\ref{one}) with $\ga =2$.  Consider a segment of three successive terms $x_{k-1}, x_k, x_{k+1}$.  On applying (\ref{one}) at vertex $x_k$, we obtain the equation:
\begin{eqnarray*}
 & & (x_{k+1}+x_{k-1}-2x_k)^2 = (x_{k-1}-x_k)^2 + (x_{k+1}-x_k)^2 \\
 & \Leftrightarrow &(x_k-x_{k-1})(x_k-x_{k+1})=0\,,
\end{eqnarray*}
so that necessarily, $x_k$ is equal to one of its neighbours.  Conversely, if every coefficient $x_k$ has at most one adjacent coefficient taking on a different value, then the sequence solves (\ref{one}) with $\ga =2$.  Clearly, for all orders $\geq 4$, non-constant cyclic sequences with $\ga = 2$ exist.

\begin{theorem} \label{thm:cyclic} {\rm (Construction of real quadratic cyclic sequences)}:  Let $q(x) = a_{n-2}x^{n-2}+a_{n-3}x^{n-3} + \cdots + a_1x + a_0$ $(n\geq 2$) be any polynomial with integer coefficients all strictly positive.  Multiply by $x+1$ to obtain the new polynomial 
\begin{eqnarray}
p(x) & : = & b_{n-1}x^{n-1} + b_{n-2}x^{n-2} + \cdots + b_1x + b_0 \label{cyclic-equ} \\
 & = & a_{n-2}x^{n-1} + (a_{n-2}+a_{n-3})x^{n-2} + \cdots + (a_1+a_0)x + a_0\,. \nonumber
\end{eqnarray}
  Let $y$ be any real root of $p(x)$ (necessarily negative).  Then a quadratic cyclic sequence 
$$
(x_0, x_1, x_2, \ldots , x_{N-1})
$$ 
of order $N = 2\sum_ka_k$ is constructed by arbitrarily prescribing $x_0$ and then requiring increments $u_{j} = x_{j+1}-x_{j}$ of successive terms to be taken from the set $\{ 1, y, y^2, \ldots , y^{n-1}\}$ in such a way that each increment $y^k$ occurs precisely $b_k$ times and any two adjacent increments have powers that differ by precisely one.  This is always possible and up to these constraints, the ordering is arbitrary.  The constant $\ga$ in {\rm (\ref{one})} is given by $\ga = 2(1+y^2)/(1-y)^2 < 2$.  

Conversely, up to addition of a constant, cyclic permutations and order reversal, a multiple of any non-constant real cyclic sequence with $\ga\neq 2$ arises in this way from such a polynomial $p(x)$, well-defined up to replacement of $p(x)$ by $\wt{p}(x) := x^{\deg p}p(1/x)$.  

The cyclic sequences with $\ga = 2$ are characterized as those made up of connected segments of order $\geq 2$ on which the sequence is constant. The cyclic sequences with $\ga = 1$ are, up to normalization, equivalent to $(0,1,0,1, \ldots , 0,1)$; they arise by taking the root $y = -1$ of $p(x)$.
\end{theorem}

We refer to the increment $y$ as a \emph{fundamental increment associated to the sequence} and the polynomial $p(x)$ as a \emph{defining polynomial of the sequence}.  As we see below $y$ is only defined up to replacement by $1/y$ and $p(x)$ up to replacement by $x^{\deg p}p(1/x)$. The ordering of increments specified by the statement of the theorem will be refered to as \emph{legitimate}.  

\begin{remark}  {\rm Since any real root $y$ must be strictly negative and adjacent powers differ by one, it follows that a real cyclic sequence with $\ga \neq 2$ oscillates.  The length $N$ of the sequence is given by $\sum_kb_k = 2\sum_ka_k$, so that a non-trivial real QCS always has even order (also a consequence of oscillation).  }
\end{remark}

\begin{example} \label{ex:real-hexagon}  {\rm Take $q(x) = x+2$. Multiplication by $x+1$ gives the polynomial $p(x) = x^2+3x+2$ with real root $y = -2$.  Arrange the powers of this root with appropriate multipicity to give the legitimate sequence of increments $(1,y,1,y,y^2,y) = (1,-2,1, -2, 4, -2)$.  We construct a real QCS of order $6$ by first setting $x_0 = 0$ and then proceeding so that $u_0 = x_1-x_0 = 1, u_1 = x_2 - x_1 = - 2$ and so on.  We thereby obtain the QCS $(0,1,-1,0, -2, 2)$ of order $6$.  We can normalize the sequence in such a way that the minimum value is $0$ and that this occurs for the first term: $(0, 4, 2, 3, 1, 2)$.

}
\end{example} 

\begin{example} {\rm  Irrational sequences arise from irrational roots.  For example, the polynomial $x^2+4x+1$ has root $x = - 2 + \sqrt{3}$.  On multiplying by $x+1$ we obtain the polynomial $p(x) = x^3+5x^2+ 5x+1$.  A legitimate sequence of increments is given by $(1,y,y^2,y,y^2,y,y^2,y,y^2,y^3,y^2,y)$ with $y = - 2 + \sqrt{3}$.  On calculating, we can now construct a real QCS of order $12$; explicitly, it is given by $(0,1,-1+\sqrt{3}, 6-3\sqrt{3}, 4-2\sqrt{3}, 11-6\sqrt{3}, 9-5\sqrt{3}, 16 - 9\sqrt{3}, -10 + 6\sqrt{3}, -3 + 2\sqrt{3}, -5 + 3\sqrt{3}, 2-\sqrt{3})$.  All terms of this QCS lie in the interval $[0,1]$.  The value of the constant $\ga$ in (\ref{one}) is given by $\ga = 4/3$. }
\end{example}

In order to prove Theorem \ref{thm:cyclic}, we first establish a recurrence relation that determines each term of the sequence as a function of the three previous terms.

\begin{lemma} \label{lem:recurrence}  Let $(x_0, x_1, \ldots , x_{N-1})$ be a non-constant real QCS  with $\ga\neq 2$, then the increments $u_j = x_{j+1} - x_{j}$ are non-zero, and satisfy $u_0+u_1+ \cdots + u_{N-1}=0$ and the recurrence relation:
$$
u_j = \left\{ \begin{array}{rl} {\rm either} \ & u_{j-1}{}^2/u_{j-2} \\
{\rm or} \ & u_{j-2}\,.
\end{array} \right.
$$
Conversely, any real sequence $(u_0, u_1, \ldots , u_{N-1})$ with non-zero terms whose sum is zero satisfying the recurrence relation, determines a real QCS.  
\end{lemma}

\begin{proof}
Let $(x_0, x_1, \ldots , x_{N-1})$ be a non-constant QCS with $\ga\neq 2$.  Consider a particular segment of the sequence consisting of four consecutive terms: $(x_{j-2}, x_{j-1}, x_j, x_{j+1})$.  On normalising, we can suppose this segment equivalent to $(u,0,v,w)$, for some real numbers $u,v,w$.  Note that $u\neq 0$, for otherwise $v$ would have to be zero since if not, we would have $\ga =2$ at term $j-1$.  But then proceeding along the cycle, we would eventually encounter a non-zero value at a vertex, which would then imply $\ga =2$ at that vertex; a contradiction.  Similarly $v \neq 0$ and more generally, two successive terms are distinct so that the increments are all non-zero. 

On evaluating equation (\ref{one}) at term $j-1$, we obtain:
\begin{equation} \label{a}
\ga = \frac{2(u^2+v^2)}{(u+v)^2} <2\,.
\end{equation}

Now evaluate (\ref{one}) at term $k$:
$$
\ga (w-2v)^2 = 2\big((w-v)^2+v^2\big)\,.
$$    
On eliminating $\ga$, we obtain the quadratic equation in $w$:
$$
uw^2+(u-v)^2w-v(u-v)^2 = 0\,.
$$
This gives two possible values $v(u-v)/u$ and $v-u$ for $w$.  To recover the general case, we set $u = x_{j-2}-x_{j-1}$, $v=x_j-x_{j-1}$, $w=x_{j+1}-x_{j-1}$.  This gives the two values:
$$
x_{j+1} = \left\{ \begin{array}{l}\ds\frac{x_{j-1}(x_j-x_{j-1})+x_j(x_{j-2}-x_j)}{x_{j-2}-x_{j-1}} \\
x_j+x_{j-1}-x_{j-2}
\end{array}
\right.
$$
On subtracting $x_j$ from both sides, we obtain the recurrence relation for the increments in the statement of the lemma. 

Conversely, given a sequence $(u_0, u_1, \ldots , u_{N-1})$ with non-zero terms whose sum is zero satisfying the recurrence relation, then the sequence must alternate in sign,  for otherwise, if we have two consecutive terms of the same sign, then all subsequent terms would have the same sign, contradicting $u_0+u_1+ \cdots + u_{N-1}=0$.
It is then easily checked that \eqref{gamma-u} is satisfied for $\gamma = 2(u_j{}^2+u_{j-1}{}^2)/(u_j-u_{j-1})^2$ independent of $j$, furthermore, the alternance in sign means that the denominator is non-zero.  A corresponding real QCS is determined by arbitrarily choosing $x_0$ and then setting $x_{j+1} = u_j+x_j$ for $j = 0, \ldots , N-1$. The condition $u_0+u_1+ \cdots + u_{N-1}=0$ means that $x_N=x_0$ and the sequence is cyclic as required. 
\end{proof}

\noindent \emph{Proof of Theorem {\rm \ref{thm:cyclic}}}.  Let $(x_0, x_1, \ldots , x_{N-1})$ be a non-constant real QCS with $\ga\neq 2$.  We show how it arises from a polynomial with positive integer coefficients.

First normalise so that $u_0 = x_1-x_0 = 1, u_1 = x_2-x_1 = y<0$ and consider the sequence of possible increments that can arise from Lemma \ref{lem:recurrence}:
\begin{equation} \label{norm-increments}
\Big( 1,y,\left\{ \begin{array}{ll} y^2, & \left\{ \begin{array}{l} y^3, \\ y, \end{array} \right. \\ 1,  & \left\{ \begin{array}{l} 1/y, \\ y, \end{array} \right. \end{array} \right. \cdots \Big)
\end{equation}
Suppose first that $y=-1$, then we obtain the sequence of increments $(1,-1,1,-1,\ldots , 1, -1)$ which corresponds to a QCS equivalent to $(0,1,0,1, \ldots , 0,1)$ with $\ga  = 1$.  Furthermore, any non-constant QCS with $\ga  = 1$ is equivalent to one of this form, since if we take a segment $(x_{j-1},  x_j,  x_{j+1})$, then
$$
(x_j-x_{j-1}+x_j-x_{j+1})^2 = 2(x_j-x_{j-1})^2 + 2(x_j-x_{j+1})^2\ \Leftrightarrow \ (x_{j+1}-x_{j-1})^2=0\,,
$$
so that $x_{j+1} = x_{j-1}$.  Henceforth, suppose that $y\neq -1$.  In particular, since $y$ is real and negative, we cannot have $y^r=1$ for any power $r\neq 0$.

The recurrence relation implies that all terms must have the form $y^r$ for some integer $r$ and that powers of successive terms differ by $\pm 1$, that is an occurrence of $y^r$ must be followed by either $y^{r +1}$ or $y^{r -1}$. Write the sequence of increments as $(y^{r_0}, y^{r_1}, \ldots , y^{r_{N-1}})$ and let $t = \min \{ r_j : 0 \leq j \leq N-1 \}$.  Now multiply through by $y^{-t}$ to obtain $1 = y^0$ in some position, with all other powers of $y$ greater than or equal to zero:
\begin{equation} \label{ysk}
(y^{s_0}, y^{s_1}, \ldots , 1 , \ldots , y^{s_{N-1}})\,,
\end{equation}
where each $s_j\geq 0$ and $s_j - s_{j-1} = \pm 1$.  
But $\sum_{j=0}^{N-1}u_j = \sum_{j=0}^{N-1}(x_{j+1}-x_{j}) = 0$, which implies that $\sum_{j=0}^{N-1}y^{s_j} = y^{-t}\sum_{j=0}^{N-1}y^{r_j}  = 0$ so that $y$ satisfies a polynomial equation of the form: 
$$
p(x):= b_{n-1}x^{n-1} + b_{n-2}x^{n-2} + \cdots + b_1x+b_0 = 0
$$
where $b_s$ is the number of occurrences of $y^s$ in \eqref{ysk} and $n-1$ is the maximal power that occurs (the choice of $n-1$ is to accord with later conventions). We can be more explicit about the form of $p(x)$. Since successive powers differ by $1$, it follows that the length of the sequence is even and that $p(-1) = 0$; thus $x+1$ is a factor. For $n\geq 4$, we claim that $p(x)$ has the following form: 
\begin{eqnarray}
p(x) & = &  (x+1)(b_{n-1}x^{n-2} + x^{n-3} + x^{n-4} + \cdots + x + b_0)  + (x+1)\sum_{s = 2}^{n-2} \be_s x^{s-1} \nonumber \\
 & = &  b_{n-1}x^{n-1} + (b_{n-1}+1)x^{n-2} + 2x^{n-3}+2x^{n-4} + \cdots \nonumber \\
  & & \qquad \qquad \cdots  + 2 x^2 + (b_0+1)x + b_0 + \sum_{s = 2}^{n-2} \be_s(x^s+x^{s-1}) \label{real-poly}
 \end{eqnarray}
 where $b_{n-1}, b_0$ are strictly positive integers and $\be_s$ are integers that are $\geq 0$. To see this, for each of the $b_{n-1}$ occurrences of the maximum power $y^{n-1}$, we must have at least one more occurrence of $y^{n-2}$.  Similarly for each of the $b_0$ occurrences of the minimum power $y^0$, we must have one more occurrence of $y^1$.  Since by assumption $y\neq -1$ so that $y^k\neq 1$ for any $k\neq 1$, all intermediate powers must occur at least twice.  However, we may have further oscillations between powers of $y^s$ and $y^{s-1}$ for $s = 2, \ldots , n-2$, which are given by the coefficients $\be_s$.  

In the case when $n=3$, then we must have $p(x) = (x+1)(a_1x+a_0)$ for positive integers $a_1,a_0$, and for $n=2$, $p(x) = a_0(x+1)$ for a positive integer $a_0$. The polynomial $p(x)$ thus has the form of the statement of the theorem.  

We need to show that the above procedure is well-defined, that is, if we perform the various operations on the QCS (addition of a constant, multiplication by a constant, cyclic permutation, order reversal), the polynomial that results is well-defined.  In fact, $p(x)$ is only defined up to replacement by $\wt{p}(x):= x^{\deg p}p(1/x)$.  

First note that all of the above operations on a QCS leave $\ga$ invariant. The expression for $\ga$ is deduced from \eqref{gamma-u} and \eqref{norm-increments}:
$$
\ga = 2\frac{(1+y^2)}{(1-y)^2}\,. 
$$
Thus, $y<0$ is determined to be a root of the quadratic equation
\begin{equation} \label{quadratic}
(1-y)^2\ga = 2(1+y^2)
\end{equation} 
and the only other root is $1/y$.  In fact, $\ga$ is invariant under  $y \mapsto y^{-1}$. Thus for given $\ga$ the only two possible fundamental increments are roots $y$ and $1/y$ of \eqref{quadratic}. 

Addition of a constant makes no difference to the sequence of increments and so leaves the above construction of $p(x)$ invariant. However, as we now show, up to a multiple, cyclic permutations and order reversal have the same effect on the sequence of increments (see Example \ref{ex:perm} below) which may modify $p(x)$.  

As before, begin with a real QCS which we normalize as above:
\begin{eqnarray*}
(x_0,x_1, \ldots , x_{N-1}) & \ra & (0,x_1-x_0, x_2-x_0, \ldots , x_{N-1} - x_0) \\
 & \ra & \left( 0, 1,\frac{x_2-x_0}{x_1-x_0}, \frac{x_3-x_0}{x_1-x_0} , \ldots , \frac{x_{N-1} - x_0}{x_1 - x_0} \right)\,,
\end{eqnarray*}
where, in the previous notation, $y = \frac{x_2-x_0}{x_1-x_0} - 1 = \frac{x_2-x_1}{x_1-x_0}$. Then the sequence of increments has the form
$$
\underline{u} := (u_0,u_1, \ldots , u_{N-1}) = \left( 1, \frac{x_2-x_1}{x_1-x_0}, \frac{x_3-x_2}{x_1-x_0}, \ldots, \frac{x_{N-1} - x_{N-2}}{x_1-x_0} , \frac{x_0-x_{N-1}}{x_1-x_0}\right)\,,
$$
where for the moment we don't normalize to make $y^0$ the smallest power of $y$.  Suppose we make a cyclic permutation to obtain the real QCS $(x_t, x_{t+1}, \ldots , x_{N-1}, x_0, x_1, \ldots , x_{t-1})$. Then the same proceedure yields the sequence of increments
\begin{eqnarray*}
\underline{v} & = & \left( 1, \frac{x_{t+2}-x_{t+1}}{x_{t+1}-x_t}, \frac{x_{t+3}-x_{t+2}}{x_{t+1}-x_t}, \ldots , \frac{x_0-x_{N-1}}{x_{t+1}-x_t}, \frac{x_1-x_0}{x_{t+1}-x_t} , \ldots , \frac{x_t-x_{t-1}}{x_{t+1} - x_t}\right) \\
 & = & \left( \frac{x_1-x_0}{x_{t+1} - x_t}\right) \left( \frac{x_{t+1}-x_t}{x_1-x_0}, \frac{x_{t+2}-x_{t+1}}{x_1-x_0}, \ldots , \frac{x_0-x_{N-1}}{x_1-x_0}, 1 , \frac{x_2-x_1}{x_1-x_0},  \ldots , \frac{x_t-x_{t-1}}{x_1-x_0} \right)
\end{eqnarray*}
which is a multiple of a cyclic permutation of $\underline{u}$. 

Similarly, order reversal gives the QCS $(x_{N-1}, x_{N-2}, \ldots , x_1, x_0)$ which, following the same normalization procedure, yields the corresponding sequence of increments
$$
\left( \frac{x_1-x_0}{x_{N-1}-x_{N-2}}\right) \left( \frac{x_{N-1}-x_{N-2}}{x_1-x_0}, \frac{x_{N-2}-x_{N-3}}{x_1-x_0}, \ldots , \frac{x_2-x_1}{x_1-x_0}, 1, \frac{x_0-x_{N-1}}{x_1-x_0}\right)
$$
which is a multiple of the sequence $\underline{u}$ with order reversed together with a cyclic permutation. In each case we obtain a multiple of $\underline{u}$ together with a cyclic permutation and/or with order reversal. 

Now proceed as before to construct the polynomial $p(x)$. Thus some multiple of $\underline{u}$ yields a sequence of powers of $y$ (respectively $1/y$), the smallest being $y^0$ (respectively $(1/y)^0$). Suppose as in \eqref{ysk}, 
$$
c\underline{u} = (y^{s_0}, y^{s_1}, \ldots , y^{s_{N-1}})\,,
$$
for some $c \in \RR \setminus \{ 0\}$ with $\min \{ s_j : 0\leq j \leq N-1\} = 0$.  Let $s_{\ell} = \max \{ s_j : 0\leq j \leq N-1\}$ and instead multiply $\underline{u}$ by $cy^{-s_{\ell}}$:
$$
cy^{-s_{\ell}} \underline{u} = (y^{s_0-s_{\ell}}, y^{s_1-s_{\ell}}, \ldots , y^{s_{N-1}-s_{\ell}}) = (\wt{y}^{s_{\ell} - s_0}, \wt{y}^{s_{\ell} - s_1}, \ldots , \wt{y}^{s_{\ell} - s_{N-1}})
$$
where $\wt{y} = 1/y$.  Note that the smallest power of $\wt{y}$ in the sequence is $\wt{y}^0$.  Furthermore, $p(x) = x^{s_0}+x^{s_1} + \cdots + x^{s_{N-1}}$, so that the polynomial satisfied by $\wt{y}$ is given by 
$$
\wt{p}(x) = x^{s_{\ell} - s_0} + x^{s_{\ell} - s_1} + \cdots + x^{s_{\ell}- s_{N-1}} = x^{s_{\ell}} p(1/x)
$$
where $s_{\ell} = \deg p(x)$.   In particular these are the only two multiples of $\underline{u}$ which can be written as powers of $y$ or $1/y$ for which the lowest power is $0$.  As already established, cyclic permutations and/or order reversal of the QCS yield a multiple of a cyclic permutation and/or order reversal of the sequence of increments $\underline{u}$. However, the number or occurrences $b_k$ of the power $y^{k}$ remains invariant by these operations (up to replacement of $y$ by $1/y$ and $p(x)$ by $x^{\deg p} p(1/x)$). In particular $p(x)$ is well-defined up to replacement by $\wt{p}(x) = x^{\deg p}p(1/x)$.  
 
 Conversely, given a polynomial $q(x) = a_{n-2}x^{n-2}+a_{n-3}x^{n-3} + \cdots + a_1x + a_0$ ($n\geq 2$) with all coefficients strictly positive, one can proceed as in the statement of the theorem, to construct a corresponding QCS.  
\hfill $\Box$

\medskip

In order to construct examples of QCS, we follow the proof of Theorem \ref{thm:cyclic}. Given the polynomial $p(x) = (x+1)q(x)$, one constructs the corresponding QCS by first constructing a legitimate sequence of increments $(y^{s_0}, y^{s_1}, \ldots , y^{s_{N-1}})$ where $s_j \in \{ 0, 1, \ldots , n+1\}$ and $s_j = s_{j-1} \pm 1$ and where the power $y^j$ occurs $b_j$ times.

\begin{example} \label{ex:one}  The example of the Introduction is obtained by taking the polynomial $3x+2$ with root $y = - 2/3$.  Multiply this by $x+1$ to obtain the polynomial $p(x) = 3x^2+5x+2$.  Now proceed as follows:
$$
\begin{array}{cl}
3x^2+5x+2 & \text{(defining polynomial)} \\
\downarrow & \\
(1,y,1,y,y^2,y,y^2,y,y^2,y) & \text{(sequence of increments)}\\
\verteq &  \\
\left( 1,-\tfrac{2}{3},1,- \tfrac{2}{3}, \tfrac{4}{9}, - \tfrac{2}{3}, \tfrac{4}{9}, - \tfrac{2}{3}, \tfrac{4}{9}, - \tfrac{2}{3} \right) & \text{(seq. of increments for given root)} \\
\downarrow &  \\
\left( 0, 1, \tfrac{1}{3}, \tfrac{4}{3}, \tfrac{2}{3}, \tfrac{10}{9}, \tfrac{4}{9}, \tfrac{8}{9}, \tfrac{2}{9}, \tfrac{2}{3} \right) & \text{(corresponding cyclic sequence)}\\
\downarrow & \\
(0,9,3,12,6,10,4,8,2,6) & \text{(normalized cyclic sequence)} 
\end{array}
$$

\m

\n The second sequence is obtained by taking a different legitimate ordering of the powers of $y$ in the above construction, namely: $(1,y,y^2,y,1,y,y^2,y,y^2,y)$ 
\end{example}

As explained in the proof of Theorem \ref{thm:cyclic}, the defining polynomial corresponding to a given sequence is not invariant under cyclic permutations. 
\begin{example} \label{ex:perm} Perform a cyclic permutation on the QCS of Example \ref{ex:one}:
$$
\begin{array}{cl}
(6,10,4,8,2,6,0,9,3,12) & \text{(sequence)} \\
 \downarrow &  \\
(4,-6,4,-6,4,-6,9,-6,9,-6) & \text{(increments)} \\
\downarrow & \\
(1,- \tfrac{3}{2}, 1, - \tfrac{3}{2} ,1, - \tfrac{3}{2}, \tfrac{9}{4}, - \tfrac{3}{2}, \tfrac{9}{4}, - \tfrac{3}{2} ) & \text{(normalization)} \\
\downarrow & \\
(1,y,1,y,1,y,y^2,y,y^2,y) 
\end{array}
$$
Now the defining polynomial is given by $(x+1)(2x+3) = x^2p(1/x)$ where $p(x)$ is the polynomial of Example \ref{ex:one}.  
\end{example}

\section{Complex algebraic quadratic cyclic sequences} \label{sec:complex}
\n We return to equation \eqref{one}, but in the first instance with $\ga : \ZZ / N \ZZ \ra [-\infty , 1]$:
\begin{equation}\label{one-complex}
\frac{\ga (j)}{2}\big( 2\phi (j) - \phi (j-1)-\phi (j+1)\big)^2 = \big(\phi (j) - \phi (j-1)\big)^2 + \big(\phi (j) - \phi (j+1)\big)^2 \qquad \forall j \in \ZZ / N \ZZ
\end{equation}
We can picture $\phi : \ZZ / N \ZZ \ra \CC$ as a closed polygon in the plane with edges the straight line segments $[\phi (j), \phi (j+1)]$.  Note that the invariance \eqref{two} still holds.  The condition $\ga (j) \leq 1$ for each $j$ is imposed as a consequence of the Cauchy-Schwarz inequality discussed in the Introduction. We allow the limiting value $\ga = 1$. A first observation is that the edges of a polygon corresponding to a solution of \eqref{one-complex} must all have the same length.  

Consider three successive non-identical terms $(\phi (j - 1), \phi (j), \phi (j+1))$. By normalization we can suppose that $\phi (j) = 0\in \CC$ and $\phi (j-1) = 1 \in \CC$.  Suppose that $\phi (j+1) = z$. At the term $j$, \eqref{one-complex} takes the form:
$$
\frac{\ga (j)}{2} (1+z)^2 = 1 + z^2\,.
$$
Suppose that $z \neq -1$. Then the requirement that $\ga (j)$ be real is equivalent to
$$
either \quad \Im (z) = 0 \quad or \quad |z| = 1\,.
$$
If $z$ is real and $z \neq \pm 1$, then $\ga (j) > 1$, which corresponds to the real case.  Otherwise $|z| = 1$.  In this case, write $z = e^{\ii \al}$.  Then
\begin{equation} \label{ga-angle}
\ga (j) = \frac{2 \cos \al}{1 + \cos \al} = \frac{2\cos \ta}{ \cos \ta - 1}
\end{equation} 
where $\ta = \pi - \al$ is the exterior angle.  The two limiting cases $\al = 0$ and $\al = \pi$ correspond to $\ga (j) = 1$ and $\ga (j) = - \infty$, respectively.  As a consequence, a complex QCS corresponds to a polygon in the plane with sides of equal length.  We are interested in such polygons defined by polynomial equations.  

Up to normalization, we can suppose the length of each edge is $1$.  Suppose henceforth that \emph{$\gamma$ is constant}, so that the exterior angle $\ta$ is uniquely defined up to sign. The fundamental increment $y$ then has the form $y = e^{\ii \ta}$ and a complex QCS has corresponding sequence of increments $(y^{s_0}, y^{s_1}, \ldots , y^{s_{N-1}})$ where $s_{j+1} = s_j \pm 1$. Up to normalization, we can suppose that $s_0 = 0$ and that $s_j \geq 0$ for all $j = 0, \ldots , N-1$. Then, since the sequence is cyclic, $y$ satisfies the polynomial equation
$$
b_{n-1}y^{n-1} + b_{n-2} y^{n-2} + \cdots + b_1y + b_0 = 0\,, 
$$
where $b_k$ is the number of occurrences of $y^k$ in the sequence of increments. Call such a QCS \emph{algebraic with turning angle $\ta$}. 

A root of a monic polynomial with integer coefficients is called an \emph{algebraic integer}. Algebraic integers exist of modulus $1$ which are not roots of unity. We will return to complex algebraic QCS arising from such increments at the end of this section.  However, in the first instance, we suppose that the sequence of increments $u_j = \phi (j+1) - \phi (j)$ is taken from a set $\{1, y, y^2, \ldots , y^{n-1}\}$ where $y = e^{2m\pi \ii /n}$, where $m$ and $n$ are relatively prime with $m<n$, so that $n$ is the smallest positive integer for which $y^n=1$.  The condition the sequence be cyclic implies $y$ must be a root of a polynomial of the form
$$
p(x) = b_{n-1}x^{n-1} + b_{n-2} x^{n-2} + \cdots + b_1x + b_0\,, 
$$
where $b_k\geq 0$ represents the number of occurences of $y^k$ in the sequence of increments. Note that it may happen that some coefficients vanish.    
Our objective is to characterize the polynomials that determine a complex algebraic QCS with increment $y = e^{2m\pi \ii /n}$. 

For $y = e^{2\pi \ii /n}$, the numbers $1,y,y^2, \ldots , y^{n-1}$ are the $n$'th roots of unity.  A root of unity $\nu$ is \emph{primitive} if $1, \nu , \nu^2, \ldots , \nu^{n-1}$ are all distinct (the order of $\nu$ is $n$).  The $n$'th cyclotomic polynomial $\Phi_n(x)$ is the polynomial whose roots are the $n$'th primitive roots of unity:
$$
\Phi_n(x):= \prod_{{}_{ {\rm gcd}(n,k) = 1}^{ \ \ 1\leq k \leq n}}\left( x - e^{2\pi \ii k/n}\right)
$$
Then $\Phi_n(x)$ is irreducible over the integers and is the minimal polynomial over the integers of $y = e^{2m\pi \ii /n}$ ($m,n$ relatively prime, $m<n$) \cite{N}.  Below, we give the first few cyclotomic polynomials.  

\begin{theorem} \label{prop:cyclo} Let $\phi$ be a complex QCS with increment $y = e^{2m\pi \ii /n}$ ($m,n$ relatively prime, $m<n$).

{\rm (i)}  When $n = 2k$ ($k\geq 2$) is even, $\phi$ is determined by a polynomial of the form 
$$
p(x) = (x+1)\Phi_n(x) q(x)
$$
where $q(x)$ is a polynomial of degree $\leq \, n-2-\deg \Phi_n(x)$ whose coefficients satisfy conditions discussed on a case by case basis below.  When $n=2$, $p(x) = a(x+1)$ for some positive integer $a$.  

{\rm (ii)}  When $n$ is odd, $\phi$ is determined by a polynomial of the form
$$
p(x) = \Phi_n(x) q(x)
$$
where $q(x)$ is a polynomial of degree $\leq \, n-1-\deg \Phi_n(x)$ whose coefficients satisfy conditions discussed on a case by case basis below.  In the case when $n$ is prime, then 
$$
p(x) = a(x^{n-1} + x^{n-2} + \cdots + x + 1)\,.
$$
for some positive integer $a$.

Conversely, polynomials of the above type (with $q(x)$ to be made precise) yield a corresponding QCS.    
\end{theorem}

\begin{proof} 

Let $p(x)$ be the defining polynomial, so that $p(y) = 0$.  Suppose that $n = 2k$ is even.  If we consider the sequence of increments, then successive increments differ by a power of one, so each occurrence of $y^s$ must be followed by either $y^{s-1}$ or $y^{s+1}$.  
If we set $y=-1$, the sequence of increments has the form $(1,-1,1, -1, \ldots , 1,-1)$ or $(-1,1,-1,1, \ldots , -1,1)$ and it follows that $p(-1) = 0$, so that $x+1$ is a factor of $p(x)$ (note that since $n$ is even, the transition from $y^{n-1}$ to $y^n = y^0$ is consistent with alternation from $-1$ to $1$ when $y = -1$; this is no longer the case when $n$ is odd). Furthermore $y$ is a root of $p(x)$ so that the minimal polynomial over the integers of $y$ must also divide $p(x)$.  However, the minimal polynomial over the integers of $y= e^{2m\pi \ii /n}$ is the $n$th cyclotomic polynomial $\Phi_n(x)$. When $n$ is odd, then similarly, the $n$'th cyclotomic polynomial must divide $p(x)$ and in the case when $n$ is prime, this is given by $x^{n-1} + x^{n-2} + \cdots + x + 1$.    
\end{proof}


For background on cyclotomic polynomials see \cite{N}. The first ones are given as follows
$$
\begin{array}{lll}
\Phi_1(x) & = & x-1 \\
\Phi_2(x) & = & x+1 \\
\Phi_3(x) & = & x^2+x+1 \\
\Phi_4(x) & = & x^2+1 \\
\Phi_5(x) & = & x^4+x^3+x^2+x+1 \\
\Phi_6(x) & = & x^2-x+1 \\
\Phi_7(x) & = & x^6+x^5+x^4+x^3+x^2+x+1 \\
\Phi_8(x) & = & x^4+1 \\
\Phi_9(x) & = & x^6+x^3+1 \\
\Phi_{10}(x) & = & x^4-x^3 +x^2-x+1 \\
\Phi_{11}(x) & = & x^{10} + x^9 + x^8 + x^7 + x^6 + x^5 + x^4+x^3+x^2+x+1 \\
\Phi_{12}(x) & = & x^4-x^2+1  \\
\Phi_{13}(x) & = & x^{12} + x^{11} + \cdots + x^2+x+1 \\
\Phi_{14}(x) & = & x^6 - x^5 + x^4 - x^3 + x^2 - x + 1 \\
\Phi_{15}(x) & = & x^8 - x^7 + x^5 - x^4 + x^3 - x + 1 
\end{array}
$$
In general, if $n$ is prime, then
$$
\Phi_n(x) = x^{n-1} + x^{n-2} + \cdots + x^2 + x + 1\,. 
$$
If $n = 2^r$ ($r>0$) then
$$
\Phi_{2^r}(x) = x^{2^{r-1}} + 1\,.
$$
If $n = 2p$ for $p$ an odd prime, then
$$
\Phi_{2p}(x) = 1 - x + x^2 - \cdots + x^{p-1}\,.
$$

 Call a sequence of powers $\underline{u} = (y^{s_0}, y^{s_1}, \ldots , y^{s_{N-1}})$ \emph{legitimate} if $s_{j+1} = s_j \pm 1 ({\rm mod}\, n$).  Call a polynomial with non-negative integer coefficients \emph{legitimate} if it determines a legitimate sequence of increments from one of its roots. First we establish an elementary recursive way to determine if a polynomial is legitimate. For ease of notation, write the sequence of increments as $y^{s_0} y^{s_1} \ldots  y^{s_{N-1}}$, i.e. omit brackets and commas.  

Suppose the sequence of increments $y^{s_0} y^{s_1}, \ldots y^{s_{N-1}}$ contains three successive terms $\ldots  y^k y^{k+1} y^k \ldots$ or $\ldots y^ky^{k-1} y^k, \ldots$.  Call such a subsequence an \emph{elementary loop based at $y^k$}. 

\begin{lemma} \label{lem:loops}  Let $p(x) = b_{n-1}x^{n-1} + b_{n-2}x^{n-2} + \cdots + b_1x+b_0$ be a legitimate polynomial and let $y^{s_0} y^{s_1} \ldots y^{s_{N-1}}$ be a corresponding legitimate sequence of increments containing an elementary loop $\ldots  y^ky^{k+1} y^k \ldots$ (respectively $\ldots  y^k y^{k-1} y^k \ldots$) based at $y^k$.  Then replacement of the loop by $y^k$  yields a legitimate sequence of increments with corresponding (legitimate) polynomial 
$$
b_{n-1}x^{n-1} + \cdots + b_{k+2}x^{k+2} + (b_{k+1}-1)x^{k+1} + (b_{k}-1)x^k + b_{k-1}x^{k-1} \cdots + b_1x+b_0\,,
$$
(respectively,
$$
b_{n-1}x^{n-1} + \cdots + b_{k+1}x^{k+1} + (b_{k}-1)x^{k} + (b_{k-1}-1)x^{k-1} + b_{k-2}x^{k-2} \cdots + b_1x+b_0\,.)
$$
Furthermore, if $p(x) = b_{n-1}x^{n-1} + b_{n-2}x^{n-2} + \cdots + b_1x+b_0$ is a legitimate polynomial with $b_k \geq  b_{k-1}, b_{k+1}$ with not all three coefficients equal, then necessarily, any corresponding legitimate sequence of increments contains a loop at $y^k$.  
\end{lemma}

\begin{proof} Suppose the sequence of increments contains a loop of the form $\ldots  y^k y^{k+1} y^k \ldots$, so that the sequence contains a subsequence of the form
$$
\ldots \left. \begin{array}{c} y^{k+1} \\ y^{k-1} \end{array} \right\}  y^k y^{k+1} y^k  \left\{ \begin{array}{c} y^{k+1} \\ y^{k-1} \end{array} \right.  \ldots
$$
Then clearly replacement of this subsequence by the subsequence 
$$
\ldots \left. \begin{array}{c} y^{k+1} \\ y^{k-1} \end{array} \right\}  y^k \left\{ \begin{array}{c} y^{k+1} \\ y^{k-1} \end{array} \right.  \ldots
$$
yields a legitimate sequence. Similarly for a loop of the form $\ldots  y^k y^{k-1} y^k \ldots$. 

Suppose that there is no loop at $y^k$. Then each passage via $y^k$ has one of the following forms: 
\begin{eqnarray*}
& & \ldots  y^{k-2}y^{k-1} y^k y^{k+1}y^{k+2} \ldots \\
& & \ldots  y^{k+2}y^{k+1} y^k y^{k-1}y^{k-2} \ldots \\
& & \ldots y^{k+2}y^{k+1}y^ky^{k+1}y^{k+2}\ldots \\
& {\rm or} & \ldots y^{k-2}y^{k-1}y^ky^{k-1}y^{k-2}\ldots
\end{eqnarray*}
which would imply that $b_{k-1}, b_{k+1} \geq b_k$, in contradiction to our hypothesis. 
\end{proof} 

If, as in the above lemma, we have $b_k \geq  b_{k-1}, b_{k+1}$ with not all three coefficients equal, without further information, it is not possible to know if a loop goes left or right, that is, takes the form  $\ldots  y^k y^{k-1} y^k \ldots$ or $\ldots  y^k y^{k+1} y^k \ldots$. For example, if all loops at $y^k$ have the form $\ldots , y^k y^{k+1} y^k\ldots$, then the following could arise with all inequalities possible:
$$
\begin{array}{ccl}
\ldots y^{k-1}y^ky^{k+1}y^ky^{k-1}\ldots  & \quad & b_k=b_{k-1} > b_{k+1} \\
\ldots y^{k-1}y^ky^{k+1}y^ky^{k+1}y^ky^{k-1} \ldots & \quad & b_k> b_{k-1} = b_{k+1}  \\
\ldots y^{k-1}y^ky^{k+1}y^ky^{k+1}y^ky^{k+1}y^ky^{k-1} \ldots & \quad & b_k>b_{k+1}> b_{k-1}
\end{array}
$$

Call the process of removing a loop, that is replacing $\ldots  y^k y^{k+1} y^k \ldots$ or $\ldots  y^k y^{k-1} y^k \ldots$ by $\ldots y^k \ldots$, \emph{reduction}. Then in order to establish the legitimacy of a polynomial we must find one reduction which yields a legitimate polynomial.  Note that there is no guarantee that a reduced polynomial (one obtained by reduction) corresponds to a quadratic sequence which is \emph{cyclic}, even if the initial polynomial does. 

\begin{example} The following sequence of reductions shows the polynomial $p(x) = 2x^5+2x^4+3x^3+2x^2+2x+3$ is legitimate. Indeed, it has root $y = e^{\pi \ii / 3}$ defining a complex QCS. 
$$
\begin{array}{c}
2x^5+2x^4+3x^3+2x^2+2x+3 \\
\downarrow \\
x^5+2x^4+2x^3+x^2+2x+2 \\
\downarrow \\
x^5+x^4+x^3+x^2+x+1
\end{array}
$$
At the first step we remove loops at $y^3$ and $y^0$, at the second we remove once again loops at $y^3$ and $y^0$. The final polynomial corresponds to the sequence $1yy^2y^3y^4y^5$ which is clearly legitimate.  A legitimate sequence for $p(x)$ is given by $1yy^2y^3y^2y^3y^4y^3y^4y^51y1y^5$. 
\end{example}

\begin{corollary} \label{cor:loops} Let $p(x) = b_{n-1}x^{n-1} + b_{n-2} x^{n-2} + \cdots + b_1x + b_0$ be a polynomial with $b_k\geq b_{k+1} + b_{k-1}$ for some $k \in \{ 0, \ldots , n-1\}$ ($k\pm 1$ calculated modulo $n$).  Then $p(x)$ is not legitimate.
\end{corollary}

\begin{proof} By Lemma \ref{lem:loops}, $p(x)$ can be reduced by removing loops at $y^k$ until the coefficient of $x^{k+1}$ (respectively $x^{k-1}$) become zero, with the coefficient of $x^k$ non-zero and greater than or equal to the coefficient of $x^{k-1}$ (respectively $x^{k+1}$). But this is not legitimate. For example, if the coefficient of $x^{k+1}$ is zero, to be legitimate, any legitimate sequence must have the form $\ldots y^{k-1}y^ky^{k-1}\ldots$ about $y^k$. But then the coefficient of $x^{k-1}$ will be strictly greater than the coefficient of $x^k$. 
\end{proof}

\m

\n \emph{Case by case analysis of the polynomials defining a complex algebraic QCS with fundamental increment $y = e^{2m\pi \ii /n}$ ($m,n$ relatively prime, $m<n$)}.  We apply Theorem \ref{prop:cyclo}.

\m

\n  $n=3$:  $p(x) = a(x^2+x+1)$, where $a$ is a positive integer.  

\m

\n  $n=4$:  $p(x) = a(x+1)(x^2+1)$, where $a$ is a positive integer.  

\m

\n   $n=5$: $p(x) = a(x^4+x^3+x^2+x+1)$, where $a$ is a positive integer. 

\m

\n  $n=6$: $p(x) = (x+1)(x^2-x+1)(ax^2+bx+c) = (x^3+1)(ax^2+bx+c) = ax^5+bx^4+cx^3+ax^2+bx+c$.  First note that the coefficients $a,b,c$ must be strictly positive so as not to disconnect powers of $y$ in the sequences of increments.  What about other constraints?  Clearly if $a=b=c>0$, then we have a legitimate sequence of increments, namely the cyclic sequence $y^0y^1y^2y^3y^4y^5$ covered $a$ times.  Suppose then that the coefficients are not all equal, for example, suppose $a \geq b,c$ with one of the inequalities strict. By Corollary \ref{cor:loops}, we cannot have $a \geq b+c$. Suppose then that $a<b+c$. We claim this is legitimate. For example if $a\geq b$ and $a > c$, then reduce the pair $ab$ successively $a-c$ times until the sequence of coefficients $abcabc$ becomes $c(b-a+c)cc(b-a+c)c$. Now reduce the two pairs $cc$ sucessively, until we have all coefficients equal to $b-a+c$ ($>0$ by hypothesis). This is now legitimate, corresponding to the cyclic sequence taken $b-a+c$ times. By symmetry we have the following characterization: \emph{the polynomial $p(x)$ is legitimate if and only if all of $a,b,c$ are positive and the maximum coefficient of $\{ a,b,c\}$ is strictly less than the sum of the other two coefficients.}

For example, take $a=b=2$ and $c=1$ to give $p(x) = 2x^5+2x^4+x^3+2x^2+2x+1$.  One possible sequence of increments has the form $(1,y,y^2,y,y^2,y^3,y^4,y^5,y^4,y^5)$ with corresponding QCS: 

\bigskip

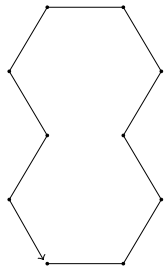
\begin{center}
\begin{tikzpicture}[line width=0.1mm,black,scale=0.5]
\draw (0,0) -- (2,0) -- (3,1.7) -- (2,3.4) -- (3,5.1) -- (2,6.8) -- (0,6.8) -- (-1,5.1) -- (0,3.4) -- (-1,1.7);
\draw [->] (-1,1.7) -- (-0.1,0.1);
\filldraw [black] (0,0) circle (1pt);
\filldraw [black] (2,0) circle (1pt);
\filldraw [black] (3,1.7) circle (1pt);
\filldraw [black] (2,3.4)  circle (1pt);
\filldraw [black] (3,5.1)  circle (1pt);
\filldraw [black] (2,6.8)  circle (1pt);
\filldraw [black] (0,6.8)  circle (1pt);
\filldraw [black] (-1,5.1)  circle (1pt);
\filldraw [black] (0,3.4)  circle (1pt);
\filldraw [black] (-1,1.7)  circle (1pt);
\end{tikzpicture}
\end{center}
\bigskip

\m

\n  $n=7$: $p(x) = a(x^6+x^5+x^4+x^3+x^2+x+1)$, where $a$ is a positive integer.  

\m

\n  $n=8$: $p(x) = (x+1)(x^4+1)(ax^2+bx+c) = ax^7+(a+b)x^6+(b+c)x^5+cx^4+ax^3+(a+b)x^2+(b+c)x+c$, where we clearly require $a,c,a+b,b+c>0$.  If $a,b,c>0$, then the polynomial is legitimate. Specifically, we first reduce the sequence of coefficients $a(a+b)(b+c)ca(a+b)(b+c)c$ to $aaccaacc$ in the obvious way and then reduce to $aaaaaaaa$ if $a \leq c$ or to $cccccccc$ if $c \leq a$.  

On the other hand, $b$ may be zero or negative.  If $b$ is zero, then $p(x) = ax^7+ax^6+cx^5+cx^4+ax^3+ax^2+cx+c$ is clearly legitimate. 
If $b$ is strictly negative and, say $a\geq c$, then this is again legitimate (similarly if $c \geq a$):  Write $e=-b>0$. Then the sequence of coefficients takes the form $a(a-e)(c-e)ca(a-e)(c-e)c$ (with $c-e>0$). Now reduce $e$ times at $y^3$ and $y^7$ to obtain the sequence of coefficients $(a-e)(a-e)(c-e)(c-e)(a-e)(a-e)(c-e)(c-e)$, which then reduces to the legitimate sequence $(c-e)(c-e)(c-e)(c-e)(c-e)(c-e)(c-e)(c-e)$. 

It is worth illustrating the above construction with an example. The polynomial $p(x) = 3x^7+2x^6+x^5+2x^4+3x^3+2x^2+x+2$ is legitimate ($a=3, b = -1, c=2$). To construct a corresponding sequence of increments, we work backwards from the cyclic sequence following the above procedure: $1yy^2y^3y^4y^5y^6y^7 \ra 1yy^2y^3y^2y^3y^4y^5y^6y^7y^6y^7 \ra 1yy^2y^3y^2y^3y^4y^3y^4y^5y^6y^7y^6y^71y^7$. The corresponding QCS is illustrated below.

\bigskip

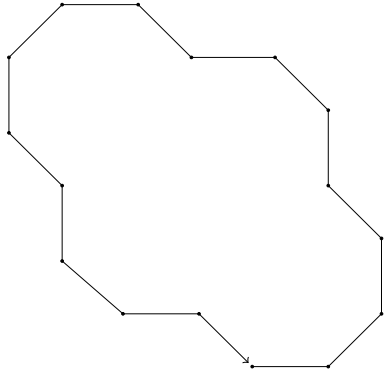
\begin{center}
\begin{tikzpicture}[line width=0.1mm,black,scale=0.5]
\draw (0,0) -- (2,0) -- (3.4,1.4) -- (3.4,3.4) -- (2,4.8) -- (2,6.8) -- (0.6,8.2) -- (-1.6,8.2) -- (-3,9.6) -- (-5,9.6) -- (-6.4,8.2) -- (-6.4,6.2) -- (-5,4.8) -- (-5,2.8) -- (-3.4,1.4) -- (-1.4,1.4);
\draw [->] (-1.4,1.4) -- (-0.1,0.1);
\filldraw [black] (0,0) circle (1pt);
\filldraw [black] (2,0) circle (1pt);
\filldraw [black] (3.4,1.4) circle (1pt);
\filldraw [black] (3.4,3.4)  circle (1pt);
\filldraw [black] (2,4.8)  circle (1pt);
\filldraw [black] (2,6.8)  circle (1pt);
\filldraw [black] (0.6,8.2)   circle (1pt);
\filldraw [black] (-1.6,8.2)  circle (1pt);
\filldraw [black] (-3,9.6)   circle (1pt);
\filldraw [black] (-5,9.6)  circle (1pt);
\filldraw [black] (-6.4,8.2)  circle (1pt);
\filldraw [black] (-6.4,6.2)  circle (1pt);
\filldraw [black] (-5,4.8) circle (1pt);
\filldraw [black] (-5,2.8) circle (1pt);
\filldraw [black] (-3.4,1.4) circle (1pt);
\filldraw [black] (-1.4,1.4) circle (1pt);
\end{tikzpicture}
\end{center}
\bigskip

To summarize: \emph{the coefficients $a,b,c$ are legitimate if and only if $a,c, a+b, b+c >0$}.

\m

\n  $n=9$: $p(x) = (x^6+x^3+1)(ax^2+bx+c) = ax^8+bx^7+cx^6+ax^5+bx^4+cx^3+ax^2+bx+c$.  This is analogeous to the case $n=6$, with the same constraints on the coefficients $a,b,c$.

\m

\n  $n=10$: $p(x) = (x+1)(x^4-x^3+x^2-x+1)(a_4x^4+a_3x^3+a_2x^2+a_1x+a_0) = (x^5+1)(a_4x^4+a_3x^3+a_2x^2+a_1x+a_0)$.  This is analogous to the case $n = 6$ but more complicated. Clearly a necessary condition is that $a_j> 0 \ \forall j$.  Rather than give an exhaustive treatment of the different cases, it suffices to apply the recursive procedure given by Lemma \ref{lem:loops}. 

\m

\n  $n = 11$: $p(x) = a(x^{10}+x^9+x^8+x^7+x^6+x^5+x^4+x^3+x^2+x+1)$ for a positive integer $a$.  

\n $n = 12$: $p(x) = (x+1)(x^4-x^2+1)(a_6x^6+a_5x^5+a_4x^4+a_3x^3+a_2x^2+a_1x+a_0)$.  This time $(x+1)\Phi_{12}(x)$ has a different form to the previous cases and we don't have an exhaustive description of the admissible coefficients.  We discuss this case further below.  

\m

The above case by case analysis exhibits certain symmetry properties that we make precise in the following proposition. In her thesis, the third author takes a different approach to these symmetry properties involving linear algebra \cite{GH}. 

\begin{proposition} \label{prop:symmetry}  Any complex algebraic QCS with turning angle $2\pi m/n$ ($m, n$  relatively prime with $m<n$) either with $n\leq 11$, $n=2^r$ ($r>0$) or $n=2p$ ($p$ an odd prime) must use all increments $\{ y^0, y^1, \ldots , y^{n-1}\}$ ($y = e^{2\pi m \ii / n}$). Furthermore if $n$ is even with the same hypotheses, then for each occurence of the increment $y^k$, there is also an occurence of the increment $y^{\frac{n}{2}+k} = - y^k$. In particular, the corresponding polygonal walk in the plane contains each edge with its oppositely orientated counterpart.  
\end{proposition}

\begin{proof}  The proposition follows from the case by case analysis above. Specifically, for $n \leq 11$, the defining polynomial $p(x)$ has the form $p(x) = a_{n-1}x^{n-1} + a_{n-2}x^{n-2} + \cdots + a_1x+a_0$ with all $a_j > 0$ for $j = 0, \ldots , n-1$.  

Furthermore, for $n = 2\ell$ even, up to $n=10$, the defining polynomial $p(x)$ always has the form $p(x) = (x^{\ell}+1)(a_{\ell -1}x^{\ell -1} + a_{\ell -2}x^{\ell -2} + \cdots + a_1x+a_0)$, so that the coefficient of $x^k$ is the same as the coefficient of $x^{\ell +k}$, for $k = 0, \ldots , \ell -1$. This continues to hold if $n = 2^r$ ($r>0$) and $n=2p$ ($p$ an odd prime). For $n=2^r$ the cyclotomic polynomial has the form $x^{2^{r-1}} + 1$. For $n=2p$, the cyclotomic polynomial is given by $x^{p-1} - x^{p-2} + \cdots - x + 1$, but $x+1$ must also be a factor of the defining polynomial, in particular $(x+1)(x^{p-1} - x^{p-2} + \cdots - x + 1) = x^p+1$ is a factor.
\end{proof}

When $n=12$, the properties of the above proposition no longer hold in general.  
Consider the cyclotomic polynomial  $\Phi_{12}(x) =x^4-x^2+1$.  This has as a root the fundamental increment $y = e^{\pi \ii /6}$ corresponding to the exterior angle of a regular polygon of $12$ sides. Let $q(x)=2x^6+2x^5+3x^4+3x^3+2x^2+2x+1$. Then 
$$
q(x) \Phi_{12}(x) = 2x^{10}+2x^9+x^8+x^7+x^6+x^5+2x^4+x^3+x^2+2x+1,
$$ 
has all coefficients strictly positive. However, it is not a legitimate polynomial. But we can now apply the construction of \S \ref{sec:rqcs} and multiply by $x+1$ to obtain the legitimate polynomial 
$$
p(x)=2x^{11}+4x^{10}+3x^9+2x^8+2x^7+2x^6+3x^5+3x^4+2x^3+3x^2+3x+1.
$$
This defines a complex algebraic QCS with sequence of increments, say 
$$
y^0y^1y^2y^3y^4y^5y^6y^7y^8y^9y^{10}y^{11}y^{10}y^{11}y^{10}y^9y^{10}y^9y^8y^7y^6y^5y^4y^5y^4y^3y^2y^1y^2y^1
$$ 
for which symmetry no longer holds (left-hand path of Fig.2). 

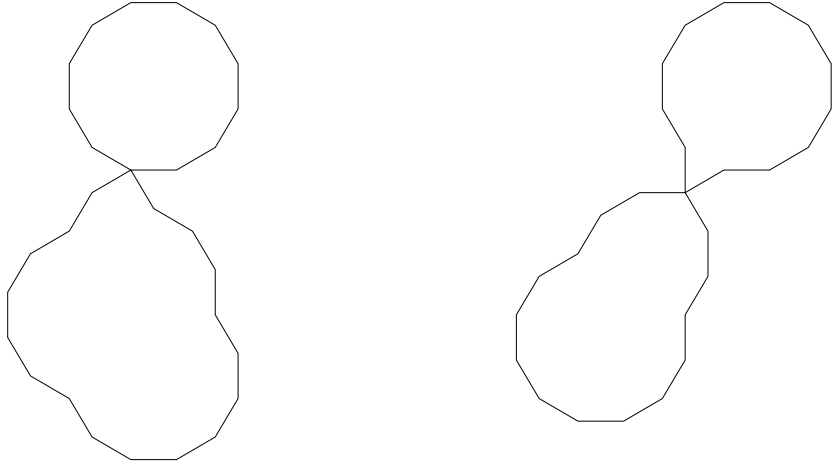
\begin{figure}[h] \label{fig:symmetry0}

\begin{tikzpicture}[line width=0.1mm,black,scale=0.3]
\draw (0,0) -- (2,0) -- (3.7,1) -- (4.7,2.7) -- (4.7,4.7) -- (3.7,6.4) -- (2, 7.4) -- (0,7.4) -- (-1.7,6.4) -- (-2.7,4.7) -- (-2.7,2.7) -- (-1.7, 1) -- (0,0) -- (1,-1.7) -- (2.7,-2.7) -- (3.7,-4.4) -- (3.7,-6.4) -- (4.7,-8.1) -- (4.7,-10.1) -- (3.7,-11.8) -- (2,-12.8) -- (0,-12.8) -- (-1.7,-11.8) -- (-2.7,-10.1) -- (-4.4,-9.1) -- (-5.4,-7.4) -- (-5.4,-5.4) -- (-4.4,-3.7) -- (-2.7,-2.7) -- (-1.7,-1) -- (0,0);
\draw (26,0) -- (28,0) -- (29.7,1) -- (30.7,2.7) -- (30.7,4.7) -- (29.7,6.4) -- (28, 7.4) -- (26,7.4) -- (24.3,6.4) -- (23.3,4.7) -- (23.3,2.7) -- (24.3,1) -- (24.3,-1) -- (25.3,-2.7) -- (25.3,-4.7) -- (24.3,-6.4) -- (24.3,-8.4) -- (23.3,-10.1) -- (21.6,-11.1) -- (19.6,-11.1) -- (17.9,-10.1) -- (16.9, -8.4) -- (16.9,-6.4) -- (17.9,-4.7) -- (19.6,-3.7) -- (20.6,-2) -- (22.3,-1) -- (24.3,-1) -- (26,0);

\end{tikzpicture}
\caption{The two figures represent walks with turning angle $\pi / 6$; the left is not symmetric; the right doesn't exploit all edges of a regular $12$-gon.}
\end{figure}

It is also the case that for $n=12$, all edges of the corresponding regular polygon are not required to complete a complex algebraic QCS with fundamental increment $y = e^{\pi \ii /6}$.  

Let $q(x)=2x^5+2x^4+3x^3+3x^2+2x+2$. Then 
$$
q(x) \Phi_{12}(x)=(2x^5+2x^4+3x^3+3x^2+2x+2)(x^4-x^2+1)= 2x^9+2x^8+x^7+x^6+x^5+x^4+x^3+x^2+2x+2
$$
is a polynomial of degree $9$ with all coefficients positive and with $y = e^{\pi \ii / 6}$ as a root.  Although is it not legitimate, multiplication by $x+1$ once more leads to the legitimate polynomial 
$$
p(x) = 2x^{10} + 4x^9+3x^8+2x^7+2x^6+2x^5+2x^4+2x^3+ 3x^2+4x+2,
$$
from which we can construct a complex algebraic QCS which doesn't use the edge $y^{11}$. Computer analysis shows that the right-hand path of Fig.2 is the smallest length path for which this property holds. Furthermore, there is no polynomial $q(x)$ of degree $\leq 4$ for which $q(x) \Phi_{12}(x)$ has strictly positive coefficients, so that for $n=12$, any corresponding QCS may omit at most one edge.  

One can proceed in an ad hoc fashion for $n>12$. One easily checks that the next case for which all edges are not required to complete a cycle is $n = 18$. The cyclotomic polynomial $\Phi_{18}(x) = x^6-x^3+1$. Muliplication by $2 x^8+2x^7+ 2x^6+ 3x^5+3x^4+3x^3+2x^2+2x+2$ yields the polynomial 
$$
2x^{14}+2x^{13}+ 2x^{12} + x^{11} + x^{10} + x^9+x^8+x^7+x^6+x^5+x^4+x^3+2x^2+2x+2
$$ 
with all coefficients positive. Multiplication by $x+1$ yields the legitimate polynomial of degree $15$ which determines a complex QCS with fundamental increment $y = e^{\pi \ii / 9}$ of length $42$ which doesn't use the edges $y^{16}$ and $y^{17}$. Two is the most number of edges that can be omitted for $n=18$ and furthermore, the above polynomial minimises the length in this case. However, if we multiply $\Phi_{18}(x)$ by $x^9+2x^8+2x^7+2x^6+3x^5+3x^4+2x^3+2x^2+2x+1$ we obtain the polynomial 
$$
x^{15} + 2x^{14} + 2x^{13} + x^{12} + x^{11} + x^{10} + x^9+x^8+x^7+x^6+x^5+x^4+x^3+2x^2+2+1
$$
with all coefficients positive. Multiplication by $x+1$ yields a legitimate polynomial of degree $16$ which determines a complex QCS of length $40$ which omits one edge.  This is the minimum length which uses $17$ edges. Thus the minimum length may decrease as one allows more edges. 

In the case $n = 30$, $\Phi_{30}(x) = x^8+x^7-x^5-x^4-x^3+x+1$. Multiplication by $4x^{13} + 2x^{12} + 2x^{11} + 3x^{10} + 4x^9+5x^8+3x^7+3x^6+5x^5+4x^4+3x^3+2x^2+2x+4$ yields the polynomial 
\begin{eqnarray*}
4x^{21} + 6x^{20} + 4x^{19} + x^{18} + x^{17} +x^{16}+x^{15}+ x^{14} + 2x^{13}+x^{12} +x^{11}+x^{10} +x^9+2x^8 \\ 
+x^7+x^6 +x^5+x^4+x^3+4x^2+6x+4\,.
\end{eqnarray*}
Multiplication by $x+1$ then gives a legitimate polynomial of degree $22$ which determines a complex QCS with fundamental increment $y = e^{\pi \ii / 15}$ of length $92$ which omits $7$ edges.  This is the maximum number of edges that can be omitted for $n=30$ and the above polynomial minimizes the length in this case.

\m

\n \emph{Complex QCS from algebraic increments which are not roots of unity}:
The smallest degree for which a unit modulus algebraic integer which is not a root of unity can occur is $4$. This can be seen by first noting that the degree must be even (roots come in reciprocal pairs), and then that any unit in a \emph{quadratic} field extension must be a root of unity.

The  palyndromic polynomial 
$$
q(x) = x^4+3x^3+3x^2+3x+1\,,
$$
irreducible over the integers, 
has two conjugate complex roots $a, \ov{a}$ and two real roots. Since for each root $b$, one must have $1/b$ a root, it follows that $\ov{a} = 1/a$ so that $|a| = 1$. However $a$ cannot be an $n$'th root of unity for any $n$, since if this were the case, $q(x)$ would divide $x^n-1$ ($q(x)$ being irreducible is the minimal polynomial of $a$ over the integers). But this would mean that the two real roots are also roots of unity, which is clearly not the case. In fact the two complex roots are given approximately by $-0.191 \pm 0.982 \, \ii$, and the two real roots, approximately by $-2.154$ and $-0.464$. Although the coefficients of $q(x)$ do not yield a corresponding sequence of increments (see below), they are all strictly positive and so, as in \S \ref{sec:rqcs}, we can multiply by $x+1$ to obtain the polynomial
$$
p(x) = (x+1)(x^4+3x^3+3x^2+3x+1) = x^5+4x^4+6x^3+6x^2+4x+1
$$
admitting, for example, the legitimate sequence of increments
$$
y^0y^1y^2y^3y^4y^5y^4y^3y^4y^3y^4y^3y^2y^3y^2y^3y^2y^1y^2y^1y^2y^1\,,
$$
where $y$ is one of the complex roots of $p(x)$. One then constructs a corresponding complex algebraic QCS. 

Another example arises from the palyndromic polynomial $x^4+2x^3+2x+1$. This has two complex conjugate roots of modulus $1$ (approx. $0.366 \pm 0.931 \ii$) which are not roots of unity by the same reasoning as above. Indeed, there are two real roots given approx. by $-2.297$ and $-0.435$.  This time the polynomial does not have all coefficients strictly positive, however, multiplication by $x+1$ yields $x^5+3x^4+2x^3+2x^2+3x+1$ which, although not legitimate, does have all coefficients positive and once more multiplying by $x+1$ yields the legitimate polynomial
$$
p(x) = x^6+4x^5+5x^4+4x^3+5x^2+4x+1\,.
$$
Taking $y$ to be one of the complex roots now yields a complex algebraic QCS with approximate turning angle $68.5$ degrees:  

\begin{figure}[H] \label{fig:non-nth}
\begin{tikzpicture}
\draw (0,0) -- (1,0) -- (1.3665, 0.93) -- (0.6365,1.612) -- (-0.2635,1.182) -- (-0.1935,0.184) -- (0.7565,-0.116) -- (1.3865,0.661) -- (2.3365,0.361) -- (2.4065,-0.637) -- (3.3565, -0.937) -- (3.4265, -1.935) -- (4.3765, -2.235) -- (4.4465,-3.233) -- (3.5465,-3.663) -- (3.6165,-4.661) -- (2.7165,-5.091) -- (1.9865,-4.409) -- (1.0865,-4.839) -- (0.3565,-4.157) -- (0.723,-3.227) -- (-0.007,-2.545) -- (0.3595,-1.615) -- (-0.37,-0.93) -- cycle; 
\node at (0.1,-0.2) {$0$};
\end{tikzpicture}
\caption{Complex algebraic QCS not arising from a root of unity}
\end{figure}
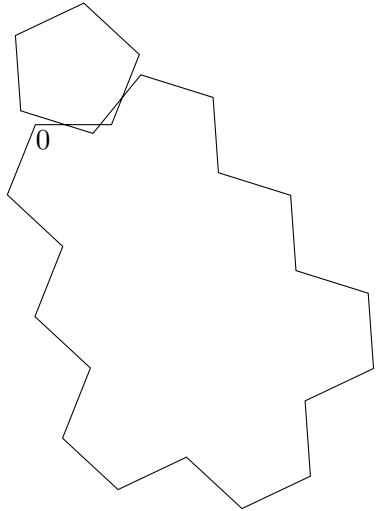

\section{Combining quadratic cyclic sequences} \label{sec:three} \n On applying the normalization \eqref{two}, cyclic sequences with common $\ga$, can be combined to form new sequences. We will refer to this construction as \emph{concatenation}. In this section, we show how this can be done for real QCS and explain how concatenation is reflected in the defining polynomials. With suitable modifications, the same procedures apply to complex algebraic QCS. We begin with an example. 

\begin{example} \label{ex:conc1} As in Example \ref{ex:one}, take the defining polynomial $(x+1)(3x+2)$ with root $y = - 2/3$ leading to the cyclic sequence \eqref{ex1}. Normalize the sequence by dividing by $9$ in order that the second term be $1$:
$$
\left( 0, 1, \tfrac{1}{3}, \tfrac{4}{3}, \tfrac{2}{3}, \tfrac{10}{9}, \tfrac{4}{9}, \tfrac{8}{9}, \tfrac{2}{9}, \tfrac{2}{3} \right)
$$
Now make another normalization of \eqref{ex1} by subtracting $6$ and dividing by $4$:
$$
\left( -\tfrac{3}{2}, \tfrac{3}{4}, - \tfrac{3}{4}, \tfrac{3}{2}, 0, 1, - \tfrac{1}{2}, \tfrac{1}{2}, - 1, 0\right)
$$
Note that in so doing, we obtain the pair $0,1$ in a different location. 
We can visualize the coefficients of the two sequences as labels on cyclic graphs:

\bigskip

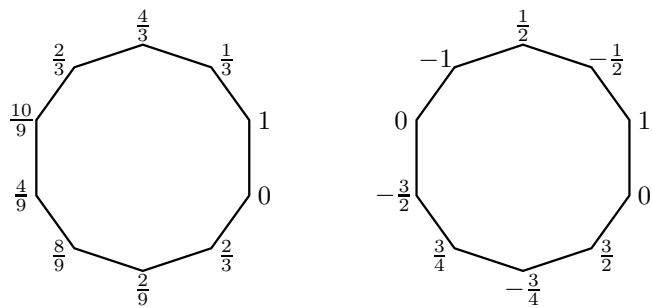
\begin{center}
\small{\begin{tikzpicture}[line width=0.3mm,black,scale=1]
\draw (0,0) -- (0.5, -0.7) -- (1.4,-1) -- (2.3,-0.7) -- (2.8,0) -- (2.8,1) -- (2.3,1.7) -- (1.4,2) -- (0.5,1.7) -- (0,1) -- cycle;
\node at (3, 0) {$0$};
\node at (3,1) {$1$};
\node at (2.5,1.8) {$\tfrac{1}{3}$};
\node at (1.4,2.25) {$\tfrac{4}{3}$};
\node at (0.3,1.8) {$\tfrac{2}{3}$};
\node at (-0.2,1) {$\tfrac{10}{9}$};
\node at (-0.2,0) {$\tfrac{4}{9}$};
\node at (0.3,-0.8) {$\tfrac{8}{9}$};
\node at (1.4,-1.25) {$\tfrac{2}{9}$};
\node at (2.5,-0.8) {$\tfrac{2}{3}$};
\draw (5,0) -- (5.5, -0.7) -- (6.4,-1) -- (7.3,-0.7) -- (7.8,0) -- (7.8,1) -- (7.3,1.7) -- (6.4,2) -- (5.5,1.7) -- (5,1) -- cycle;
\node at (8, 0) {$0$};
\node at (8,1) {$1$};
\node at (7.5,1.8) {$-\tfrac{1}{2}$};
\node at (6.4,2.25) {$\tfrac{1}{2}$};
\node at (5.25,1.8) {$-1$};
\node at (4.8,1) {$0$};
\node at (4.7,0) {$-\tfrac{3}{2}$};
\node at (5.3,-0.8) {$\tfrac{3}{4}$};
\node at (6.4,-1.25) {$-\tfrac{3}{4}$};
\node at (7.5,-0.8) {$\tfrac{3}{2}$};
\end{tikzpicture}  }
\end{center}
\bigskip

Now switch edges as indicated to obtain a cyclic sequences of order $20$:

\bigskip

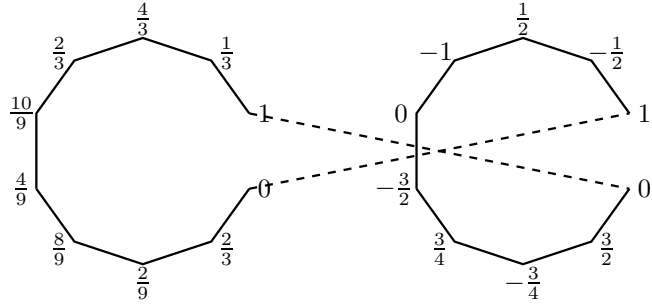
\begin{center}
\small{\begin{tikzpicture}[line width=0.3mm,black,scale=1]
\draw [dashed] (2.8,0) -- (7.8,1);
\draw [dashed] (2.8,1) -- (7.8,0);
\draw (0,0) -- (0.5, -0.7) -- (1.4,-1) -- (2.3,-0.7) -- (2.8,0);
\draw (2.8,1) -- (2.3,1.7) -- (1.4,2) -- (0.5,1.7) -- (0,1) -- (0,0);
\node at (3, 0) {$0$};
\node at (3,1) {$1$};
\node at (2.5,1.8) {$\tfrac{1}{3}$};
\node at (1.4,2.25) {$\tfrac{4}{3}$};
\node at (0.3,1.8) {$\tfrac{2}{3}$};
\node at (-0.2,1) {$\tfrac{10}{9}$};
\node at (-0.2,0) {$\tfrac{4}{9}$};
\node at (0.3,-0.8) {$\tfrac{8}{9}$};
\node at (1.4,-1.25) {$\tfrac{2}{9}$};
\node at (2.5,-0.8) {$\tfrac{2}{3}$};
\draw (5,0) -- (5.5, -0.7) -- (6.4,-1) -- (7.3,-0.7) -- (7.8,0);
 \draw (7.8,1) -- (7.3,1.7) -- (6.4,2) -- (5.5,1.7) -- (5,1) -- (5,0);
\node at (8, 0) {$0$};
\node at (8,1) {$1$};
\node at (7.5,1.8) {$-\tfrac{1}{2}$};
\node at (6.4,2.25) {$\tfrac{1}{2}$};
\node at (5.25,1.8) {$-1$};
\node at (4.8,1) {$0$};
\node at (4.7,0) {$-\tfrac{3}{2}$};
\node at (5.3,-0.8) {$\tfrac{3}{4}$};
\node at (6.4,-1.25) {$-\tfrac{3}{4}$};
\node at (7.5,-0.8) {$\tfrac{3}{2}$};
\end{tikzpicture}  }
\end{center}
\bigskip

$$
\left(  0, 1, - \tfrac{1}{2}, \tfrac{1}{2}, - 1, 0, -\tfrac{3}{2}, \tfrac{3}{4}, - \tfrac{3}{4}, \tfrac{3}{2},  0, 1, \tfrac{1}{3}, \tfrac{4}{3}, \tfrac{2}{3}, \tfrac{10}{9}, \tfrac{4}{9}, \tfrac{8}{9}, \tfrac{2}{9}, \tfrac{2}{3} \right)
$$
In order to deduce the defining polynomial, first multiply by the smallest common multiple of the denominators, namely $36$:
$$
( 0, 36, -18, 18, -36, 0, -54, 27, - 27, 54, 0, 36, 12, 48, 24, 40, 16, 32, 8, 24 )
$$
Now construct the sequences of increments $u_k = x_{k+1} - x_{k}$:
\small{\begin{eqnarray*}
 & & ( 36, -54, 36, -54, 36, -54, 81, -54, 81, -54, 36, -24, 36, -24, 16, -24, 16, -24, 16, -24) \\
& = & ( 2^23^2, -2 \times 3^3, 2^23^2, -2 \times 3^3, 2^23^2, -2 \times 3^3, 3^4, -2 \times 3^3, 3^4, \\ 
 & & \qquad \qquad -2 \times 3^3, 2^23^2, -2^3\times 3, 2^23^2, -2^3\times 3, 2^4, -2^3\times 3, 2^4, -2^3\times 3, 2^4, -2^3\times 3)
\end{eqnarray*}}
A multiple of this sequence puts the coefficients into the form $y^{\ell}$ ($\ell \geq 0$) for a fundamental increment $y$. One can take this multiple  to be $1/3^4$ with $y = -2/3$:
$$
(y^2,y,y^2,y,y^2,y,1,y,1,y,y^2,y^3,y^2,y^3,y^4,y^3,y^4,y^3,y^4,y^3)
$$
with corresponding polynomial 
$$
3x^4+5x^3+5x^2+5x+2 = (x+1)(x^2+1)(3x+2) = (x^2+1) p(x)\,.
$$ 
\end{example}

An alternative construction is to combine two sequences of different lengths with common fundamental increment.  

\begin{example}  Consider the two defining polynomials $p_1(x) = (x+1)(x+2)(2x+1)$ and $p_2(x) = (x+1)(x+2)$ with common root $y = -2$: 
$$
\begin{array}{ccc}
p_1(x) = 2x^3+7x^2+7x+2 &  & p_2(x) = x^2+3x+2 \\
\downarrow & & \downarrow \\
(y^3,y^2,y^3,y^2,y,1,y,1,y,y^2,y,y^2,y,y^2,y,y^2,y,y^2) & \text{(increments)} & (y^2,y,1,y,1,y) \\
\downarrow & & \downarrow \\
(-8,4,-8,4,-2,1,-2,1,-2,4,-2,4,-2,4,-2,4,-2,4) & (y = -2) &  (4,-2,1,-2,1,-2) \\
\downarrow & & \downarrow \\
(0,-8,-4,-12,-8,-10,-9,-11,-10,  & & \\
\qquad \qquad -12, -8, -10,-6,-8,-4,-6,-2,-4) & \text{(sequence)} & (0,4,2,3,1,2) \\
\downarrow & \downarrow \\
(0,4,2,6,4,5,\tfrac{9}{2}, \tfrac{11}{2}, 5,6,4,5,3,4,2,3,1,2) & (\times \ - \tfrac{1}{2}) &  
\end{array}
$$  
We now concatenate by placing one sequence after the other:
$$
(0,4,2,6,4,5,\tfrac{9}{2}, \tfrac{11}{2}, 5,6,4,5,3,4,2,3,1,2,0,4,2,3,1,2)
$$
By \S \ref{sec:rqcs}, some multiple of the corresponding sequence of increments yields a sequence of powers of one of the fundamental increments. In this case, multiplication by $-2$ yields the sequence of increments
$$
 (y^3,y^2,y^3,y^2,y,1,y,1,y,y^2,y,y^2,y,y^2,y,y^2,y,y^2,y^3,y^2,y,y^2,y,y^2)
$$
with defining polynomial $3x^3+10x^2+9x+2 = (x+1)(x+2)(3x+1) = p_1(x) + xp_2(x)$.  The different ordering of the increments $(y^3, y^2, y^3, y^2, y, y^2, y, 1, y, 1, y, 1, y, 1, y, y^2, y, y^2, y, y^2, y, y^2, y, y^2)$ produces the polynomial $(x+1)(x+2)(2x+2) = p_1(x) + p_2(x)$.  
\end{example} 

The above examples illustrate a more general property. 

\begin{proposition} \label{prop:splicing}  Let $p_1(x)$ and $p_2(x)$ be polynomials of degree $m-1$ and $n-1$ (resp.) of the form \emph{\eqref{cyclic-equ}} of \emph{Theorem \ref{thm:cyclic}} which define real QCS of orders $M$ and $N$ (resp.) deriving from a common fundamental increment $y$.  Then there exists a real QCS of order $M+N$ with defining polynomial $p(x)= p_1(x)+x^kp_2(x)$, where $0 \leq k \leq m$, obtained by concatenation of two real QCS with defining polynomials $p_1(x)$ and $p_2(x)$ resp.. 
\end{proposition}  

\begin{proof} Let $(y^{s_0}, y^{s_1}, \ldots , y^{s_{M-1}})$ and $(y^{t_0}, y^{t_1}, \ldots , y^{t_{N-1}})$ be legitimate sequences of increments associated to real QCS deriving from $p_1(x)$ and $p_2(x)$ resp., where we suppose $s_0=0, s_1 = 1,\ldots , s_{m-2} = m-2, s_{m-1} = m-1, s_{m} = m-2, \ldots , s_{M-1} = 1$ and $t_0=0, t_1 = 1,\ldots , t_{n-2} = n-2, t_{n-1} = n-1, t_{n} = n-2, \ldots , t_{N-1} = 1$. That is, the powers are initially monotone increasing from $0$ to $m-1$ (resp. $n-1$). This is always possible, see \eqref{inc-monot} of \S\ref{sec:euler}.

 If $y = -1$, then the result is clear - we simply concatenate two oscillating sequences of the form $(0,1,0,1, \ldots , 0,1)$.  

A legitimate sequence of increments for $p(x)$ is given by 
$$
(y^{s_0}, y^{s_1}, \ldots , y^{s_{k-1}}, y^{k+t_0}, y^{k+t_1}, \ldots , y^{k+t_{N-1}}, y^{s_k}, y^{s_{k+1}}, \ldots , y^{s_{M-1}})
$$  
with corresponding QCS:
\begin{eqnarray*}
(0,y^{s_0}, y^{s_0} + y^{s_1}, \ldots , y^{s_0} + y^{s_1} + \cdots + y^{s_{k-1}}, y^{s_0} + y^{s_1} + \cdots + y^{s_{k-1}} + y^{k+t_0},  \\
  y^{s_0} + y^{s_1} + \cdots + y^{s_{k-1}} + y^{k+t_0} + y^{k+t_1}, \\
 \ldots , y^{s_0} + y^{s_1} + \cdots + y^{s_{k-1}} + \underbrace{y^{k+t_0} + y^{k+t_1} + \cdots + y^{k+t_{N-1}}}_{0}, \\
y^{s_0} + y^{s_1} + \cdots + y^{s_{k-1}} + \underbrace{y^{k+t_0} + y^{k+t_1} + \cdots + y^{k+t_{N-1}}}_{0} + y^{s_k}, \ldots  ) 
\end{eqnarray*}
Since $(y^{t_0}, y^{t_1}, \ldots , y^{t_{N-1}})$ is a legitimate sequence of increments associated to $p_2(x)$, we have $\sum_{j = 0}^{N-1}y^{t_j} = 0 \ \Rightarrow \ \sum_{j = 0}^{N-1}y^{k+t_j}=0$. It follows that the pairs of successive terms 
$$
( y^{s_0} + y^{s_1} + \cdots + y^{s_{k-1}}, y^{s_0} + y^{s_1} + \cdots + y^{s_{k-1}} + y^{s_k})
$$
and 
$$
\begin{array}{l}
( y^{s_0} + y^{s_1} + \cdots + y^{s_{k-1}} + y^{k+t_0} + y^{k+t_1} + \cdots + y^{k+t_{N-1}}, \\ 
\qquad y^{s_0} + y^{s_1} + \cdots + y^{s_{k-1}} + y^{k+t_0} + y^{k+t_1} + \cdots + y^{k+t_{N-1}} + y^{s_k} )
\end{array}
$$
coincide and the sequence can be obtained by concatenation of two sequences of orders $M$ and $N$, with defining polynomials $p_1(x)$ and $p_2(x)$, resp.. 
\end{proof}

\section{Eulerian digraphs} \label{sec:euler}

\n In order to better understand the correspondence between polynomials and QCS, notably the legitimate sequences of increments that can arise, it is useful to model the collection of increments with an Eulerian digraph.  

A digraph is a pair $D = (V,A)$ consisting of a (finite) set of vertices $V$ and an abstract set $A$ together with a map $f : A \ra V \times V$ -- the (oriented) arcs.  In general multiple arcs and loops are allowed. The \emph{order} of $D$ is the cardinality of $V$.  When $a \in A$ corresponds to an arc from $x$ to $y$, it is sometimes useful to write $a = xy$.  For each vertex $x \in V$, denote by $d^-(x)$ the in-degree at $x$ and by $d^+(x)$ the out-degree.  If, for all $x \in V$, $d^-(x) = d^+(x)$, the digraph is called a \emph{balanced digraph}.  

A directed walk in a digraph $D = (V,A)$ is a sequence $v_0a_1v_1a_2\cdots a_kv_k$ where $v_j \in V$, $a_j \in A$ and $a_j = v_{j-1}v_j$, $1 \leq j \leq k$, with no arc repeated.  A digraph is said to be \emph{Eulerian} if it contains a closed directed walk which traverses every arc of $D$ exactly once.  Eulerian digraphs are characterized by the following theorem \cite{B-JG}.  

\begin{theorem}  A digraph $D = (V,A)$ is Eulerian if and only if $D$ is connected and for each of its vertices $x$, $d^-(x) = d^+(x)$.
\end{theorem}

If the defining polynomial of a QCS has degree $n-1$, then the digraphs we will use as a model will have vertices $\{ 0,1,2, \ldots , n-1\}$ and arcs only of the form $j(j-1)$ or $(j-1)j$ where $j$ and $j-1$ are taken modulo $n$.  Call such a digraph a \emph{$1$-step digraph}.  We will view the vertices as arranged in cyclic order.  For example, take
$$
p(x) = (x+1)(2x^4 + 4x^3 + x^2+2x+2) = 2x^5+6x^4+5x^3+3x^2+4x+2
$$
(with real root $-2$).  We construct a corresponding Eulerian digraph with $22$ edges as follows.

\bigskip

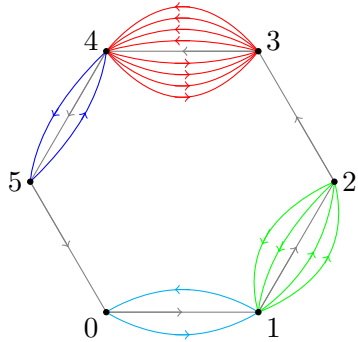
\begin{center}

\begin{tikzpicture}
\draw [gray] (0,0) -- (2,0) -- (3,1.73) -- (2,3.46) -- (0, 3.46) -- (-1,1.73) -- cycle;
\draw [->,gray] (0,0) -- (1,0);
\draw [->,gray] (2,0) -- (2.5,0.865);
\draw [->,gray] (3,1.73) -- (2.5,2.595);
\draw [->,gray] (2,3.46) -- (1,3.46);
\draw [->,gray] (0,3.46) -- (-0.5,2.595);
\draw [->,gray] (-1,1.73) -- (-0.5,0.865);
\draw [cyan] plot [smooth, tension=1] coordinates { (0,0) (1,-0.3) (2,0)};
\draw [->,cyan] (0.9,-0.3) -- (1.1,-0.3);
\draw [cyan] plot [smooth, tension=1] coordinates { (0,0) (1,0.3) (2,0)};
\draw [->,cyan] (1.1,0.3) -- (0.9,0.3);
\draw [green] plot [smooth, tension=1] coordinates { (2,0) (2.7,0.75) (3,1.73)};
\draw [green] plot [smooth, tension=1] coordinates { (2,0) (2.9,0.7) (3,1.73)};
\draw [green] plot [smooth, tension=1] coordinates { (2,0) (2.3,0.87) (3,1.73)};
\draw [green] plot [smooth, tension=1] coordinates { (2,0) (2.1,1) (3,1.73)};
\draw [red] plot [smooth, tension=1] coordinates { (2,3.46) (1,3.6) (0,3.46)};
\draw [red] plot [smooth, tension=1] coordinates { (2,3.46) (1,3.75) (0,3.46)};
\draw [red] plot [smooth, tension=1] coordinates { (2,3.46) (1,3.9) (0,3.46)};
\draw [red] plot [smooth, tension=1] coordinates { (2,3.46) (1,4.05) (0,3.46)};
\draw [red] plot [smooth, tension=1] coordinates { (2,3.46) (1,3.3) (0,3.46)};
\draw [red] plot [smooth, tension=1] coordinates { (2,3.46) (1,3.15) (0,3.46)};
\draw [red] plot [smooth, tension=1] coordinates { (2,3.46) (1,3.0) (0,3.46)};
\draw [red] plot [smooth, tension=1] coordinates { (2,3.46) (1,2.85) (0,3.46)};
\draw [blue] plot [smooth, tension=1] coordinates { (0,3.46) (-0.7,2.595) (-1,1.73)};
\draw [blue] plot [smooth, tension=1] coordinates { (0,3.46) (-0.3,2.595) (-1,1.73)};
\draw [->,green] (2.675,0.7) -- (2.725,0.8);
\draw [->,green] (2.875,0.65) -- (2.925,0.75);
\draw [->,green] (2.325,0.94) -- (2.27,0.8);
\draw [->,green] (2.15,1.1) -- (2.05,0.9);
\draw [->,red] (1.1,3.6) -- (0.9,3.6);
\draw [->,red] (1.1,3.75) -- (0.9,3.75);
\draw [->,red] (1.1,3.9) -- (0.9,3.9);
\draw [->,red] (1.1,4.05) -- (0.9,4.05);
\draw [->,red] (0.9,3.3) -- (1.1,3.3);
\draw [->,red] (0.9,3.15) -- (1.1,3.15);
\draw [->,red] (0.9,3.0) -- (1.1,3.0);
\draw [->,red] (0.9,2.85) -- (1.1,2.85);
\draw [->,blue] (-0.63,2.7) -- (-0.7,2.595);
\draw [->,blue] (-0.325,2.55) -- (-0.275,2.65);
\filldraw [black] (0,0) circle (1pt);
\filldraw [black] (2,0) circle (1pt);
\filldraw [black] (3,1.73) circle (1pt);
\filldraw [black] (2,3.46) circle (1pt);
\filldraw [black] (0, 3.46) circle (1pt);
\filldraw [black] (-1,1.73) circle (1pt);
\node at (-0.2,-0.2) {$0$};
\node at (2.2,-0.2) {$1$};
\node at (3.2,1.73)  {$2$};
\node at (2.2,3.6)  {$3$};
\node at (-0.2,3.6)  {$4$};
\node at (-1.2,1.73)  {$5$};
\end{tikzpicture}
\end{center}
\bigskip

For any vertex $j$, both the in-degree and the out-degree is equal to the coefficient of $x^j$. Call an \emph{elementary closed directed walk between $j$ and $j+1$}, one of the form $ja(j+1)bj$ ($a,b \in A$, $a\neq b$). For the above example, we have first constructed the cyclic digraph and then added elementary closed directed walks as necessary to correspond to the coefficients.  
As we see below there may be many non-isomorphic $1$-step Eulerian digraphs associated to a given polynomial. 

One can easily see that a digraph is Eulerian if and only if, after the removal of a closed directed walk, each of the connected components that remain are Eulerian, where we consider an isolated vertex as Eulerian.  Thus we can recognize the above polynomial as being legitimate, if there is a corresponding associated $1$-step Eulerian digraph obtained by removing elementary closed directed walks as follows:
$$
\begin{array}{rcl}
 p(x)= 2x^5+6x^4+5x^3+3x^2+4x+2 & \rightarrow & x^5+5x^4+5x^3+3x^2+4x+2 \\
 & \rightarrow & x^5+x^4+x^3+3x^2+4x+2  \\
 & \rightarrow & x^5+x^4+x^3+x^2+2x+2  \\
 & \rightarrow & x^5+x^4+x^3+x^2+x+1
\end{array}
$$
On the first step, we remove an elementary closed walk between $4$ and $5$; on the second step we remove $4\times$ an elementary closed walk between $3$ and $4$; on the third step we remove $2\times$ an elementary closed walk between $1$ and $2$; on the fourth step we remove an elementary closed walk between $0$ and $1$.  The end polynomial now has corresponding $1$-step Eulerian digraph, the cyclic digraph.  The procedure of removing elementary closed walks is the analogue of reduction of \S\ref{sec:complex}. However, to check that a polynomial is not legitimate by this method means checking all possible reductions of all possible Eulerian digraphs. 

For a balanced digraph of order $n$, its degree sequence is the sequence $(a_0,a_1, \ldots , a_{n-1})$ where $a_j$ is the in-degree ($=$ out-degree) of vertex $j$.  To such a digraph, we associate the polynomial $p(x) = a_{n-1}x^{n-1} + a_{n-2}x^{n-2} + \cdots + a_1+a_0$.  

Given the degree sequence for the defining polynomial of a QCS, an associated Eulerian $1$-step digraph may not be unique.  For example the polynomial $p(x) = 2+2x+2x^2+2x^3$ has three realizations:

\bigskip

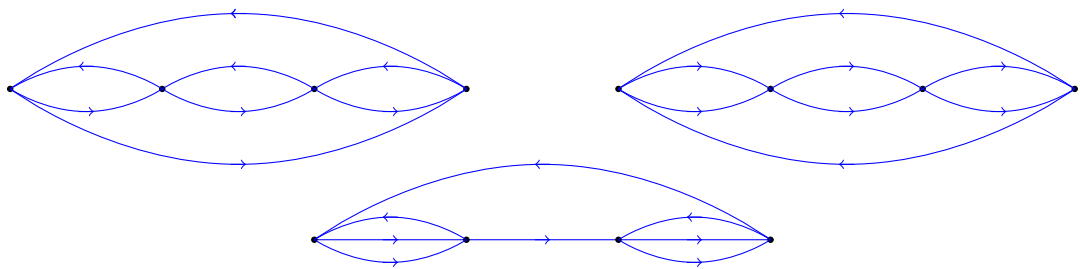
\begin{center}
\begin{tikzpicture}
\draw [->,blue] (3.1,1) -- (2.9,1);
\draw [->,blue] (0.9,-0.3) -- (1.1,-0.3);
\draw [->,blue] (1.1,0.3) -- (0.9,0.3);
\draw [->,blue] (2.9,-0.3) -- (3.1,-0.3);
\draw [->,blue] (2.9,-1) -- (3.1,-1);
\draw [->,blue] (3.1,0.3) -- (2.9,0.3);
\draw [->,blue] (3.1,1) -- (2.9,1);
\draw [->,blue] (4.9,-0.3) -- (5.1,-0.3);
\draw [->,blue] (5.1,0.3) -- (4.9,0.3);
\filldraw [black] (0,0) circle (1pt);
\filldraw [black] (2,0) circle (1pt);
\filldraw [black] (4,0) circle (1pt); 
\filldraw [black] (6,0) circle (1pt);
\draw [blue] plot [smooth, tension=1] coordinates { (0,0) (1,-0.3) (2,0)};
\draw [blue] plot [smooth, tension=1] coordinates { (0,0) (1,0.3) (2,0)};
\draw [blue] plot [smooth, tension=1] coordinates { (2,0) (3,-0.3) (4,0)};
\draw [blue] plot [smooth, tension=1] coordinates { (2,0) (3,0.3) (4,0)};
\draw [blue] plot [smooth, tension=1] coordinates { (4,0) (5,-0.3) (6,0)};
\draw [blue] plot [smooth, tension=1] coordinates { (4,0) (5,0.3) (6,0)};
\draw [blue] plot [smooth, tension=1] coordinates { (0,0) (3,1) (6,0)};
\draw [blue] plot [smooth, tension=1] coordinates { (0,0) (3,-1) (6,0)};
\draw [->,blue] (11.1,1) -- (10.9,1);
\draw [->,blue] (8.9,-0.3) -- (9.1,-0.3);
\draw [->,blue] (8.9,0.3) -- (9.1,0.3);
\draw [->,blue] (10.9,-0.3) -- (11.1,-0.3);
\draw [->,blue] (11.1,-1) -- (10.9,-1);
\draw [->,blue] (10.9,0.3) -- (11.1,0.3);
\draw [->,blue] (12.9,-0.3) -- (13.1,-0.3);
\draw [->,blue] (12.9,0.3) -- (13.1,0.3);
\filldraw [black] (8,0) circle (1pt);
\filldraw [black] (10,0) circle (1pt);
\filldraw [black] (12,0) circle (1pt); 
\filldraw [black] (14,0) circle (1pt);
\draw [blue] plot [smooth, tension=1] coordinates { (8,0) (9,-0.3) (10,0)};
\draw [blue] plot [smooth, tension=1] coordinates { (8,0) (9,0.3) (10,0)};
\draw [blue] plot [smooth, tension=1] coordinates { (10,0) (11,-0.3) (12,0)};
\draw [blue] plot [smooth, tension=1] coordinates { (10,0) (11,0.3) (12,0)};
\draw [blue] plot [smooth, tension=1] coordinates { (12,0) (13,-0.3) (14,0)};
\draw [blue] plot [smooth, tension=1] coordinates { (12,0) (13,0.3) (14,0)};
\draw [blue] plot [smooth, tension=1] coordinates { (8,0) (11,1) (14,0)};
\draw [blue] plot [smooth, tension=1] coordinates { (8,0) (11,-1) (14,0)};
\draw [->,blue] (4.9,-2.3) -- (5.1,-2.3);
\draw [->,blue] (5.1,-1.7) -- (4.9,-1.7);
\draw [->,blue] (7.1,-1) -- (6.9,-1);
\draw [->,blue] (8.9,-2.3) -- (9.1,-2.3);
\draw [->,blue] (9.1,-1.7) -- (8.9,-1.7);
\draw[blue] (4,-2) -- (6,-2);
\draw [->,blue] (4.9,-2) -- (5.1,-2);
\draw[blue] (6,-2) -- (8,-2);
\draw [->,blue] (6.9,-2) -- (7.1,-2);
\draw[blue] (8,-2) -- (10,-2);
\draw [->,blue] (8.9,-2) -- (9.1,-2);
\filldraw [black] (4,-2) circle (1pt);
\filldraw [black] (6,-2) circle (1pt);
\filldraw [black] (8,-2) circle (1pt); 
\filldraw [black] (10,-2) circle (1pt);
\draw [blue] plot [smooth, tension=1] coordinates { (4,-2) (5,-2.3) (6,-2)};
\draw [blue] plot [smooth, tension=1] coordinates { (4,-2) (5,-1.7) (6,-2)};
\draw [blue] plot [smooth, tension=1] coordinates { (8,-2) (9,-2.3) (10,-2)};
\draw [blue] plot [smooth, tension=1] coordinates { (8,-2) (9,-1.7) (10,-2)};
\draw [blue] plot [smooth, tension=1] coordinates { (4,-2) (7,-1) (10,-2)};
\end{tikzpicture}
\end{center}
\bigskip
The top left-hand one contains elementary closed walks, whereas the top right-hand one contains no elementary closed walk.  In this case the underlying multigraphs (the multigraph with the same vertex and edge set, but now with each edge undirected) are identical.  However, the lower $1$-step Eulerian digraph has underlying graph non-isomorphic to the top two.  Corresponding legitimate sequences of increments are given by $1yy^2y^31y^3y^2y$ (top left), $1yy^2y^31yy^2y^3$ (top right), $1y1yy^2y^3y^2y^3$ (bottom).

We can exploit Eulerian digraphs to see that the polynomial 
$$
p(x) = (x+1)(a_{n-2}x^{n-2}+a_{n-3}x^{n-3} + \cdots + a_1x+a_0) = a_{n-2}x^{n-1} + (a_{n-2}+a_{n-3})x^{n-2} + \cdots + (a_1+a_0)x + a_0\,,
$$
of Theorem \ref{thm:cyclic} is legitimate. Note however, that in the real case, we are not allowed to connect vertex $n-1$ with vertex $0$. 

$\bullet$ Construct the directed edges $01,12,\ldots , (n-2)(n-1), (n-1)(n-2)$ (producing an elementary closed walk between $n-2$ and $n-1$).

$\bullet$ Construct $a_{n-2} -1$ elementary closed walks between $n-2$ and $n-1$. 

$\bullet$ Construct $(n-2)(n-3)$.

$\bullet$ Construct $a_{n-3}-1$ elementary closed walks between $n-3$ and $n-2$. 

$\bullet$ Construct $(n-3)(n-4)$.

etc.

$\bullet$ Construct $10$.

$\bullet$ Construct $a_0-1$ elementary closed walks between $0$ and $1$.

Note in particular that this shows it is always possible (in the notation of \S \ref{sec:rqcs}) to begin the sequence of increments \begin{equation} \label{inc-monot}
(1,y,y^2, \ldots , y^{n-2}, y^{n-1}, y^{n-2}, \ldots  )\,.
\end{equation}

When enumerating all possible sequences of increments associated to a defining polynomial, we must consider all Eulerian $1$-step digraphs associated to the polynomial (we don't distinguish between walks which take a different arc joining the same two vertices). 

Concatenation of \S\ref{sec:three} can be represented in terms of corresponding digraphs.  Consider the last example of \S\ref{sec:three}, with $p_1(x) = 2x^3+7x^2+7x+2$ and $p_2(x) = x^2+3x+2$.  Each of these has corresponding Eulerian digraphs given by the above algorithm as illustrated. 

\bigskip

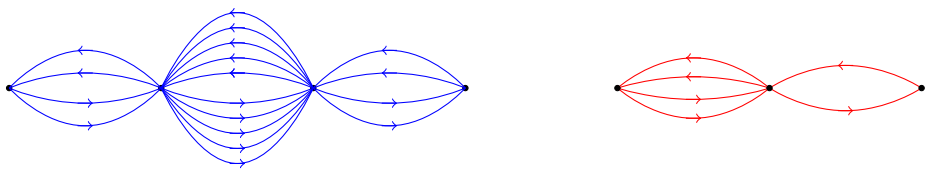
\begin{center}
\begin{tikzpicture}
\draw [->,blue] (0.9,-0.2) -- (1.1,-0.2);
\draw [->,blue] (1.1,0.2) -- (0.9,0.2);
\draw [->,blue] (0.9,-0.5) -- (1.1,-0.5);
\draw [->,blue] (1.1,0.5) -- (0.9,0.5);
\draw [->,blue] (2.9,-0.2) -- (3.1,-0.2);
\draw [->,blue] (2.9,-0.4) -- (3.1,-0.4);
\draw [->,blue] (2.9,-0.6) -- (3.1,-0.6);
\draw [->,blue] (2.9,-0.8) -- (3.1,-0.8);
\draw [->,blue] (2.9,-1) -- (3.1,-1);
\draw [->,blue] (3.1,0.2) -- (2.9,0.2);
\draw [->,blue] (3.1,0.2) -- (2.9,0.2);
\draw [->,blue] (3.1,0.4) -- (2.9,0.4);
\draw [->,blue] (3.1,0.6) -- (2.9,0.6);
\draw [->,blue] (3.1,0.8) -- (2.9,0.8);
\draw [->,blue] (3.1,1) -- (2.9,1);
\draw [->,blue] (4.9,-0.2) -- (5.1,-0.2);
\draw [->,blue] (5.1,0.2) -- (4.9,0.2);
\draw [->,blue] (4.9,-0.5) -- (5.1,-0.5);
\draw [->,blue] (5.1,0.5) -- (4.9,0.5);
\filldraw [black] (0,0) circle (1pt);
\filldraw [black] (2,0) circle (1pt);
\filldraw [black] (4,0) circle (1pt); 
\filldraw [black] (6,0) circle (1pt);
\draw [blue] plot [smooth, tension=1] coordinates { (0,0) (1,-0.2) (2,0)};
\draw [blue] plot [smooth, tension=1] coordinates { (0,0) (1,-0.5) (2,0)};
\draw [blue] plot [smooth, tension=1] coordinates { (0,0) (1,0.2) (2,0)};
\draw [blue] plot [smooth, tension=1] coordinates { (0,0) (1,0.5) (2,0)};
\draw [blue] plot [smooth, tension=1] coordinates { (2,0) (3,-0.2) (4,0)};
\draw [blue] plot [smooth, tension=1] coordinates { (2,0) (3,-0.4) (4,0)};
\draw [blue] plot [smooth, tension=1] coordinates { (2,0) (3,-0.6) (4,0)};
\draw [blue] plot [smooth, tension=1] coordinates { (2,0) (3,-0.8) (4,0)};
\draw [blue] plot [smooth, tension=1] coordinates { (2,0) (3,-1) (4,0)};
\draw [blue] plot [smooth, tension=1] coordinates { (2,0) (3,0.2) (4,0)};
\draw [blue] plot [smooth, tension=1] coordinates { (2,0) (3,0.4) (4,0)};
\draw [blue] plot [smooth, tension=1] coordinates { (2,0) (3,0.6) (4,0)};
\draw [blue] plot [smooth, tension=1] coordinates { (2,0) (3,0.8) (4,0)};
\draw [blue] plot [smooth, tension=1] coordinates { (2,0) (3,1) (4,0)};
\draw [blue] plot [smooth, tension=1] coordinates { (4,0) (5,-0.2) (6,0)};
\draw [blue] plot [smooth, tension=1] coordinates { (4,0) (5,0.2) (6,0)};
\draw [blue] plot [smooth, tension=1] coordinates { (4,0) (5,-0.5) (6,0)};
\draw [blue] plot [smooth, tension=1] coordinates { (4,0) (5,0.5) (6,0)};
\draw [red] plot [smooth, tension=1] coordinates { (8,0) (9,-0.15) (10,0)};
\draw [red] plot [smooth, tension=1] coordinates { (8,0) (9,-0.4) (10,0)};
\draw [red] plot [smooth, tension=1] coordinates { (8,0) (9,0.15) (10,0)};
\draw [red] plot [smooth, tension=1] coordinates { (8,0) (9,0.4) (10,0)};
\draw [red] plot [smooth, tension=1] coordinates { (10,0) (11,-0.3) (12,0)};
\draw [red] plot [smooth, tension=1] coordinates { (10,0) (11,0.3) (12,0)};
\draw [->,red] (8.9,-0.15) -- (9.1,-0.15);
\draw [->,red] (8.9,-0.4) -- (9.1,-0.4);
\draw [->,red] (9.1,0.15) -- (8.9,0.15);
\draw [->,red] (9.1,0.4) -- (8.9,0.4);
\draw [->,red] (10.9,-0.3) -- (11.1,-0.3);
\draw [->,red] (11.1,0.3) -- (10.9,0.3);
\filldraw [black] (8,0) circle (1pt);
\filldraw [black] (10,0) circle (1pt);
\filldraw [black] (12,0) circle (1pt);
\end{tikzpicture}
\end{center}
\bigskip

When we concatenate the two sequences, we obtain a sequence with defining polynomial $3x^3+10x^2+9x+2$, with (one of different possible) corresponding Eulerian digraph:
\bigskip

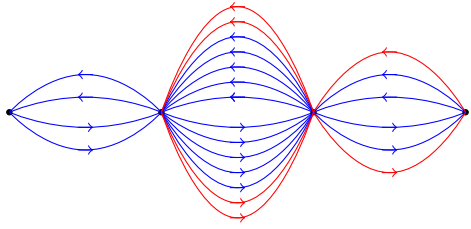
\begin{center}
\begin{tikzpicture}
\draw [->,blue] (0.9,-0.2) -- (1.1,-0.2);
\draw [->,blue] (1.1,0.2) -- (0.9,0.2);
\draw [->,blue] (0.9,-0.5) -- (1.1,-0.5);
\draw [->,blue] (1.1,0.5) -- (0.9,0.5);
\draw [->,blue] (2.9,-0.2) -- (3.1,-0.2);
\draw [->,blue] (2.9,-0.4) -- (3.1,-0.4);
\draw [->,blue] (2.9,-0.6) -- (3.1,-0.6);
\draw [->,blue] (2.9,-0.8) -- (3.1,-0.8);
\draw [->,blue] (2.9,-1) -- (3.1,-1);
\draw [->,blue] (3.1,0.2) -- (2.9,0.2);
\draw [->,blue] (3.1,0.2) -- (2.9,0.2);
\draw [->,blue] (3.1,0.4) -- (2.9,0.4);
\draw [->,blue] (3.1,0.6) -- (2.9,0.6);
\draw [->,blue] (3.1,0.8) -- (2.9,0.8);
\draw [->,blue] (3.1,1) -- (2.9,1);
\draw [->,blue] (4.9,-0.2) -- (5.1,-0.2);
\draw [->,blue] (5.1,0.2) -- (4.9,0.2);
\draw [->,blue] (4.9,-0.5) -- (5.1,-0.5);
\draw [->,blue] (5.1,0.5) -- (4.9,0.5);
\filldraw [black] (0,0) circle (1pt);
\filldraw [black] (2,0) circle (1pt);
\filldraw [black] (4,0) circle (1pt); 
\filldraw [black] (6,0) circle (1pt);
\draw [blue] plot [smooth, tension=1] coordinates { (0,0) (1,-0.2) (2,0)};
\draw [blue] plot [smooth, tension=1] coordinates { (0,0) (1,-0.5) (2,0)};
\draw [blue] plot [smooth, tension=1] coordinates { (0,0) (1,0.2) (2,0)};
\draw [blue] plot [smooth, tension=1] coordinates { (0,0) (1,0.5) (2,0)};
\draw [blue] plot [smooth, tension=1] coordinates { (2,0) (3,-0.2) (4,0)};
\draw [blue] plot [smooth, tension=1] coordinates { (2,0) (3,-0.4) (4,0)};
\draw [blue] plot [smooth, tension=1] coordinates { (2,0) (3,-0.6) (4,0)};
\draw [blue] plot [smooth, tension=1] coordinates { (2,0) (3,-0.8) (4,0)};
\draw [blue] plot [smooth, tension=1] coordinates { (2,0) (3,-1) (4,0)};
\draw [blue] plot [smooth, tension=1] coordinates { (2,0) (3,0.2) (4,0)};
\draw [blue] plot [smooth, tension=1] coordinates { (2,0) (3,0.4) (4,0)};
\draw [blue] plot [smooth, tension=1] coordinates { (2,0) (3,0.6) (4,0)};
\draw [blue] plot [smooth, tension=1] coordinates { (2,0) (3,0.8) (4,0)};
\draw [blue] plot [smooth, tension=1] coordinates { (2,0) (3,1) (4,0)};
\draw [blue] plot [smooth, tension=1] coordinates { (4,0) (5,-0.2) (6,0)};
\draw [blue] plot [smooth, tension=1] coordinates { (4,0) (5,0.2) (6,0)};
\draw [blue] plot [smooth, tension=1] coordinates { (4,0) (5,-0.5) (6,0)};
\draw [blue] plot [smooth, tension=1] coordinates { (4,0) (5,0.5) (6,0)};
\draw [red] plot [smooth, tension=1] coordinates { (2,0) (3,-1.2) (4,0)};
\draw [red] plot [smooth, tension=1] coordinates { (2,0) (3,-1.4) (4,0)};
\draw [red] plot [smooth, tension=1] coordinates { (2,0) (3,1.2) (4,0)};
\draw [red] plot [smooth, tension=1] coordinates { (2,0) (3,1.4) (4,0)};
\draw [red] plot [smooth, tension=1] coordinates { (4,0) (5,-0.8) (6,0)};
\draw [red] plot [smooth, tension=1] coordinates { (4,0) (5,0.8) (6,0)};
\draw [->,red] (4.9,-0.8) -- (5.1,-0.8);
\draw [->,red] (5.1,0.8) -- (4.9,0.8);
\draw [->,red] (2.9,-1.2) -- (3.1,-1.2);
\draw [->,red] (2.9,-1.4) -- (3.1,-1.4);
\draw [->,red] (3.1,1.2) -- (2.9,1.2);
\draw [->,red] (3.1,1.4) -- (2.9,1.4);
\end{tikzpicture}
\end{center}
\bigskip

Similarly, we can reverse the process of concatenation by removing an Eulerian sub-digraph.  For example, if, from the left-hand digraph corresponding to the polynomial $p(x) = 2x^3+7x^2+7x+2$, we remove an Eulerian digraph corresponding to the right-hand digraph corresponding to the polynomial $p(x) = x^2+3x+2$, we obtain the Eulerian digraph:
\bigskip

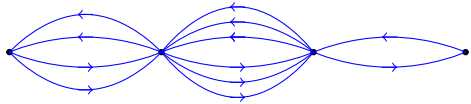
\begin{center}
\begin{tikzpicture}
\draw [->,blue] (0.9,-0.2) -- (1.1,-0.2);
\draw [->,blue] (1.1,0.2) -- (0.9,0.2);
\draw [->,blue] (0.9,-0.5) -- (1.1,-0.5);
\draw [->,blue] (1.1,0.5) -- (0.9,0.5);
\draw [->,blue] (2.9,-0.2) -- (3.1,-0.2);
\draw [->,blue] (2.9,-0.4) -- (3.1,-0.4);
\draw [->,blue] (2.9,-0.6) -- (3.1,-0.6);
\draw [->,blue] (3.1,0.2) -- (2.9,0.2);
\draw [->,blue] (3.1,0.2) -- (2.9,0.2);
\draw [->,blue] (3.1,0.4) -- (2.9,0.4);
\draw [->,blue] (3.1,0.6) -- (2.9,0.6);
\draw [->,blue] (4.9,-0.2) -- (5.1,-0.2);
\draw [->,blue] (5.1,0.2) -- (4.9,0.2);
\filldraw [black] (0,0) circle (1pt);
\filldraw [black] (2,0) circle (1pt);
\filldraw [black] (4,0) circle (1pt); 
\filldraw [black] (6,0) circle (1pt);
\draw [blue] plot [smooth, tension=1] coordinates { (0,0) (1,-0.2) (2,0)};
\draw [blue] plot [smooth, tension=1] coordinates { (0,0) (1,-0.5) (2,0)};
\draw [blue] plot [smooth, tension=1] coordinates { (0,0) (1,0.2) (2,0)};
\draw [blue] plot [smooth, tension=1] coordinates { (0,0) (1,0.5) (2,0)};
\draw [blue] plot [smooth, tension=1] coordinates { (2,0) (3,-0.2) (4,0)};
\draw [blue] plot [smooth, tension=1] coordinates { (2,0) (3,-0.4) (4,0)};
\draw [blue] plot [smooth, tension=1] coordinates { (2,0) (3,-0.6) (4,0)};
\draw [blue] plot [smooth, tension=1] coordinates { (2,0) (3,0.2) (4,0)};
\draw [blue] plot [smooth, tension=1] coordinates { (2,0) (3,0.4) (4,0)};
\draw [blue] plot [smooth, tension=1] coordinates { (2,0) (3,0.6) (4,0)};
\draw [blue] plot [smooth, tension=1] coordinates { (4,0) (5,-0.2) (6,0)};
\draw [blue] plot [smooth, tension=1] coordinates { (4,0) (5,0.2) (6,0)};
\end{tikzpicture}
\end{center}
\bigskip
with corresponding polynomial $x^3+4x^2+5x+2 = (x+2)(x+1)^2$.  However, care needs to be taken, since we may lose the root $-2$ defining the QCS.  If we remove another isomorphic copy of the same digraph, depending on how this is done, we arrive either at the digraph on the left with polynomial $x^2+3x+2 = (x+1)(x+2)$ or the digraph on the right with polynomial $x^3+2x^2+2x+1 = (x+1)(x^2+x+1)$ whose only real root is $-1$:

\bigskip

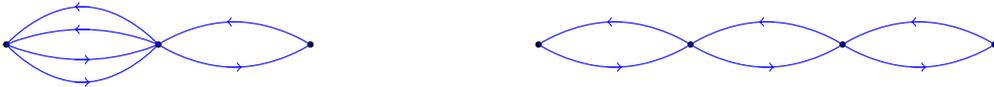
\begin{center}
\begin{tikzpicture}
\draw [->,blue] (3.9,-0.3) -- (4.1,-0.3);
\draw [->,blue] (4.1,0.3) -- (3.9,0.3);
\draw [->,blue] (1.9,-0.2) -- (2.1,-0.2);
\draw [->,blue] (2.1,0.2) -- (1.9,0.2);
\draw [->,blue] (1.9,-0.5) -- (2.1,-0.5);
\draw [->,blue] (2.1,0.5) -- (1.9,0.5);
\filldraw [black] (1,0) circle (1pt);
\filldraw [black] (3,0) circle (1pt);
\filldraw [black] (5,0) circle (1pt); 
\draw [blue] plot [smooth, tension=1] coordinates { (1,0) (2,-0.2) (3,0)};
\draw [blue] plot [smooth, tension=1] coordinates { (1,0) (2,-0.5) (3,0)};
\draw [blue] plot [smooth, tension=1] coordinates { (1,0) (2,0.2) (3,0)};
\draw [blue] plot [smooth, tension=1] coordinates { (1,0) (2,0.5) (3,0)};
\draw [blue] plot [smooth, tension=1] coordinates { (3,0) (4,-0.3) (5,0)};
\draw [blue] plot [smooth, tension=1] coordinates { (3,0) (4,0.3) (5,0)};
\draw [->,blue] (8.9,-0.3) -- (9.1,-0.3);
\draw [->,blue] (9.1,0.3) -- (8.9,0.3);
\draw [->,blue] (10.9,-0.3) -- (11.1,-0.3);
\draw [->,blue] (11.1,0.3) -- (10.9,0.3);
\draw [->,blue] (12.9,-0.3) -- (13.1,-0.3);
\draw [->,blue] (13.1,0.3) -- (12.9,0.3);
\filldraw [black] (8,0) circle (1pt);
\filldraw [black] (10,0) circle (1pt);
\filldraw [black] (12,0) circle (1pt); 
\filldraw [black] (14,0) circle (1pt);
\draw [blue] plot [smooth, tension=1] coordinates { (8,0) (9,-0.3) (10,0)};
\draw [blue] plot [smooth, tension=1] coordinates { (8,0) (9,0.3) (10,0)};
\draw [blue] plot [smooth, tension=1] coordinates { (10,0) (11,-0.3) (12,0)};
\draw [blue] plot [smooth, tension=1] coordinates { (10,0) (11,0.3) (12,0)};
\draw [blue] plot [smooth, tension=1] coordinates { (12,0) (13,-0.3) (14,0)};
\draw [blue] plot [smooth, tension=1] coordinates { (12,0) (13,0.3) (14,0)};
\end{tikzpicture}
\end{center}
\bigskip

Note that, if the polynomial $p(x)$ has an associated $1$-step Eulerian digraph $D = (V,A)$, then the polynomial $\wt{p}(x) = x^{\deg p}p(1/x)$ has an associated $1$-step Eulerian digraph $\wt{D} = (\wt{V}, \wt{A})$ which is isomorphic to $D$.  Here, isomorphism between digraphs means that the underlying multigraphs are isomorphic in a way which preserves the orientation of edges.  In fact, if $p(x) = a_{n-1}x^{n-1}+a_{n-2}x^{n-2} + \cdots + a_1x+a_0$, then $\wt{p}(x) = a_0x^{n-1}+a_1x^{n-2} + \cdots + a_{n-2}x + a_{n-1}$.   An isomorphism from $D$ to $\wt{D}$ is given by mapping vertex $j$ to $n-1-j$ ($0\leq j \leq n-1$) and mapping an edge $j(j+1)$ to $(n-j-1)(n-j-2)$.  



\section{Quadratic cyclic sequences and planar walks}  \label{sec:random}

\noindent  Let us first review the construction of a QCS, both real and complex algebraic.  Upon normalization (provided that $\ga \neq 2$), the QCS can be put into the form $(0,1,x_2, \ldots )$.  The sequence of increments is given
$$
(1, y, \left\{ \begin{array}{c} y^2 \\ 1 \end{array}\right. , \left\{ \begin{array}{l} y^3 \\ y \\ y^{-1} \end{array} \right., \ldots )
$$
where $y = x_2-1$. 
At each successive step, the increment $y^s$ is multiplied either by $y$, or by $y^{-1}$.  Thus, the normalized QCS has the form
$$
(0,1,1+y, 1+y+ \left\{ \begin{array}{c} y^2 \\ 1 \end{array}\right. , \ldots )
$$
where $y$ is the root of a polynomial whose coefficients are non-negative integers.  

Let us now remove the requirement that the sequence be cyclic.  Suppose that at each successive step, the increment $y^s$ changes according to \emph{either} $y^s \mapsto y^{s+1}$ \emph{or} $y^s \mapsto y^{s-1}$, with equal probability $1/2$.   This generates a random walk either along the real line, or in the complex plane according as to whether $y$ is real or complex, respectively.  Consider such a sequence with $y = e^{\ii \ta}$. Thus at each step we turn either right or left through an angle $\ta$ -- we call this the \emph{turning angle} of the walk. We are particularly interested in the cases $\ta = 2 \pi / n$ when $n = 4$ or $n = 6$, for then the $2$-step walk (the walk obtained by combining two successive steps) corresponds to standard walks on the square lattice, or triangular lattice, resp. 

Consider a walk with an even number of steps, where at each step we are obliged to turn left or right with turning angle $ \pi / 2$ and where we take two steps at a time. We will refer to this as a \emph{$2$-step walk with turning angle $\pi /2$}. We start at the origin $(0,0)$ and set off in one of four directions, i.e. at the first step we arrive at one of $(0, \pm 1)$ or $(\pm 1, 0)$ with equal probability $1/4$.  

\begin{lemma} After an even number of steps, we arrive at $(k,\ell )$ with $k+\ell \in 2 \ZZ$.
\end{lemma}
\begin{proof}
By induction on the number of steps $2t$.  After two steps we arrive at one of $(\pm 1, \pm 1)$.  Then each 2-step iteration replaces $(k,\ell)$ by $(k\pm 1, \ell\pm 1)$.  
\end{proof}

The \emph{standard planar walk} is a walk on the integer lattice, for which, if one is at position $(r,s )$ one moves to one of $(r\pm 1, s)$ or $(r,s \pm 1)$.  

\begin{lemma} The $2$-step planar walk with turning angle $\pi /2$ determines a standard planar walk. Conversely, a standard planar walk corresponds to precisely two $2$-step planar walks with turning angle $\pi /2$. 
\end{lemma} 
\begin{proof}
The walk is transformed into the standard walk by the mapping
$$
(k,\ell ) \mapsto \left( \frac{k+\ell}{2}, \frac{\ell -k}{2} \right)
$$
Note that, since $k+\ell \in 2 \ZZ$, the right-hand side belongs to $\ZZ^2$.  If we set $r = (k+\ell )/2$ and $s= (\ell -k)/2$, then the possible outcomes of a 2-step walk (below left) map to the possible outcomes of the standard walk (below right): 
$$
\left. \begin{array}{ccc}
(k+1, \ell +1) \mapsto (r+1,s) \\
(k+1, \ell -1) \mapsto (r,s- 1) \\
(k-1, \ell +1) \mapsto (r, s+ 1) \\
(k-1, \ell -1) \mapsto (r - 1, s) 
\end{array} \right\} \ \text{standard walk}
$$
For each 2-step $(k,\ell ) \mapsto (k \pm 1, \ell \pm 1)$ there are precisely two 1-step routes.  For example, if $(k,\ell ) \mapsto (k+1, \ell -1)$, then this is achieved by either $(k,\ell ) \mapsto (k+1,\ell ) \mapsto (k+1,\ell -1)$ or $(k,\ell ) \mapsto (k,\ell -1) \mapsto (k+1, \ell -1)$.  However, which of these two occurs is determined uniquely be the preceeding step.  Thus, given a standard walk, after an initial choice is made (of two possiblities), the 2-step walk is determined.  
\end{proof} 

The diagram below gives a standard walk (blue) and one of the two possible corresponding 2-step walks (red).

\bigskip

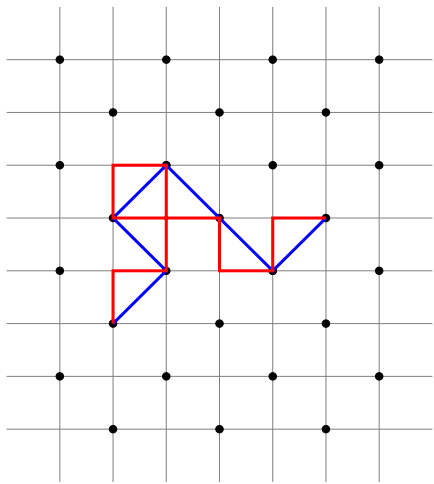
\begin{center}
\begin{tikzpicture}[line width=0.1mm,black,scale=0.7]
\draw [gray] (-2,-2) -- (6,-2);
\draw [gray] (-2,-1) -- (6,-1);
\draw [gray] (-2,0) -- (6,0);
\draw [gray] (-2,1) -- (6,1);
\draw [gray] (-2,2) -- (6,2);
\draw [gray] (-2,3) -- (6,3);
\draw [gray] (-2,4) -- (6,4);
\draw [gray] (-2,5) -- (6,5);
\draw [gray] (-1,-3) -- (-1,6);
\draw [gray] (0,-3) -- (0,6);
\draw [gray] (1,-3) -- (1,6);
\draw [gray] (2,-3) -- (2,6);
\draw [gray] (3,-3) -- (3,6);
\draw [gray] (4,-3) -- (4,6);
\draw [gray] (5,-3) -- (5,6);
\filldraw [black] (1,1) circle (2pt);
\filldraw [black] (2,2) circle (2pt);
\filldraw [black] (3,3) circle (2pt);
\filldraw [black] (4,4) circle (2pt);
\filldraw [black] (5,5) circle (2pt);
\filldraw [black] (1,3) circle (2pt);
\filldraw [black] (1,5) circle (2pt);
\filldraw [black] (1,-1) circle (2pt);
\filldraw [black] (3,1) circle (2pt);
\filldraw [black] (5,1) circle (2pt);
\filldraw [black] (-1,1) circle (2pt);
\filldraw [black] (2,4) circle (2pt);
\filldraw [black] (2,0) circle (2pt);
\filldraw [black] (0,2) circle (2pt);
\filldraw [black] (4,2) circle (2pt);
\filldraw [black] (3,5) circle (2pt);
\filldraw [black] (3,1) circle (2pt);
\filldraw [black] (3,-1) circle (2pt);
\filldraw [black] (5,3) circle (2pt);
\filldraw [black] (1,3) circle (2pt);
\filldraw [black] (-1,3) circle (2pt);
\filldraw [black] (4,0) circle (2pt);
\filldraw [black] (0,4) circle (2pt);
\filldraw [black] (5,-1) circle (2pt);
\filldraw [black] (-1,5) circle (2pt);
\filldraw [black] (-1,-1) circle (2pt);
\filldraw [black] (0,0) circle (2pt);
\filldraw [black] (0,-2) circle (2pt);
\filldraw [black] (2,-2) circle (2pt);
\filldraw [black] (4,-2) circle (2pt);
\draw [line width=0.4mm, blue] (0,0) -- (1,1);
\draw [line width=0.4mm, blue] (1,1) -- (0,2);
\draw [line width=0.4mm, blue] (0,2) -- (1,3) -- (2,2) -- (3,1) -- (4,2);
\draw [line width=0.4mm, red] (0,0) -- (0,1) -- (1,1) -- (1,2) -- (0,2) -- (0,3) -- (1,3) -- (1,2) -- (2,2) -- (2,1) -- (3,1) -- (3,2) -- (4,2);
\end{tikzpicture}  
\end{center}
\bigskip

If we set $y = \ii$, then the 2-step walk is given by the following sequence of increments $(\ii , 1, \ii , -1, \ii , 1, - \ii , 1,$ $ - \ii , 1,  \ii , 1)$.  The alternative choice of 2-step walk is given by $(1, \ii , -1, \ii , 1, \ii , 1,$ $ - \ii , 1, - \ii , 1, \ii )$.  In terms of $y$, these are given by $(y, 1,y,  y^2, y, 1, y^3, 1, y^3, 1, y, 1)$ and $(1, y, y^2, y,$ $ 1, y, 1, y^3, 1, y^3, 1, y)$, respectively.  These correspond to different legitimate arrangements of the sequence of increments with defining polynomial $p(x) = 2x^3 + x^2 + 4x+5 = (x+1)(2x^2 - x + 5)$.  Note that since the sequence is not cyclic, $y$ is not a root of $p(x)$.   

If we are in a particular position in a standard walk (at one of the orange nodes below), then in whatever direction we have arrived at that position in a 2-step walk, there is equal probabity $1/4$ of arriving after two steps at one of the adjacent nodes in the standard walk, as illustrated in the diagram below, where we suppose we arrive along the horizontal arrow coming from the right.    

\bigskip

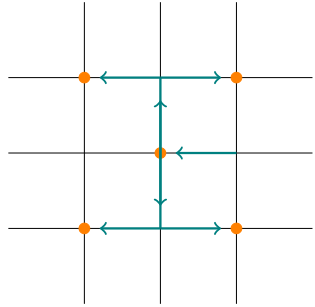
\begin{center}
\begin{tikzpicture}[line width=0.1mm,black,scale=1]
\draw (1,0) -- (1,4);
\draw (2,0) -- (2,4);
\draw (3,0) -- (3,4);
\draw (0,1) -- (4,1);
\draw (0,2) -- (4,2);
\draw (0,3) -- (4,3);
\filldraw [orange] (2,2) circle (2pt);
\filldraw [orange] (1,1) circle (2pt);
\filldraw [orange] (1,3) circle (2pt);
\filldraw [orange] (3,1) circle (2pt);
\filldraw [orange] (3,3) circle (2pt);
\draw [line width=0.3mm, teal,->] (3,2) -- (2.2,2);
\draw [line width=0.3mm, teal] (2,2) -- (2,3);
\draw [line width=0.3mm, teal] (2,2) -- (2,1);
\draw [line width=0.3mm, teal, ->] (2,3) -- (2.8,3);
\draw [line width=0.3mm, teal, ->] (2,3) -- (1.2,3);
\draw [line width=0.3mm, teal, ->] (2,1) -- (2.8,1);
\draw [line width=0.3mm, teal, ->] (2,1) -- (1.2,1);
\draw [line width=0.3mm, teal,->] (2,2) -- (2,2.7);
\draw [line width=0.3mm, teal,->] (2,2) -- (2,1.3);
\end{tikzpicture}  
\end{center}
\bigskip

Every standard planar walk of finite length determines a polynomial of the form $p(x) = b_3x^3 + b_2x^2+b_1x + b_0$.  This is obtained by choosing one of the two corresponding 2-step walks and setting $b_k$ to be the cardinality of $\ii^k$ in the sequence of increments.  Then $b_0$ and $b_2$ correspond to horizontal increments and $b_1$ and $b_3$ to vertical increments.  As for the example above, $x+1$ must always be a factor of this polynomial, since each horizonal increment must be matched by a vertical increment.  What polynomials $p(x)$ can arise from such a walk?  In what follows, we will identify the lattice of the standard walk with the points $(k,\ell )$ in the plane with $k+\ell$ even. 

\begin{lemma}  \label{lem:2step-poly} Any polynomial $p(x)$ of the form $p(x) = (x+1)(a_2x^2+a_1x+a_0)$ with $a_0, a_2 \geq 0$ and $- a_1 \leq \min\{a_0,a_2\}$ determines a standard planar walk and conversely, each standard planar walk determines such a polynomial.  The length of the standard planar walk is given by $a_0+a_1+a_2$. The walk is closed if and only if $p(x) = c(x+1)(x^2+1)$ for some positive integer $c$.  The end point of the walk is given by $p(\ii )$.   
\end{lemma}
\begin{proof} Set $p(x) = b_3x^3+b_2x^2+b_1x+b_0 = (x+1)(a_2x^2+a_1x+a_0)$ where the $a_j$ satisfy the conditions of the statement of the lemma.  These conditions are equivalent to $b_k \geq 0$ for $k = 0, \ldots , 3$.  Since $x+1$ is a factor, $-b_3+b_2-b_1+b_0 = 0$ so that the number of horizonal increments $b_0+b_2$ (given by $y^0=1, y^2=-1$  with $y = \ii$) is equal to the number of vertical increments $b_1+b_3$.  Clearly these can be ordered (non-uniquely in general) to give a 2-step walk determining a standard walk.  

For the converse, given a standard planar walk, if $b_k$ is the number of occurences of $y^k$ ($k = 0,1,2,3$, $y = \ii$) as an increment in one of the two corresponding 2-step walks, then $p(x) = b_3x^3+b_2x^2+b_1x+b_0$ is a polynomial with the desired properties.  The length of the standard planar walk is given by $(b_0+b_1+b_2+b_3)/2 = a_0+a_1+a_2$.  

If the walk is closed, i.e. it ends at its starting point, then $y=\ii$ is a root of $p(x)$.  Since the coefficients of $p(x)$ are real, $- \ii$ must also be a root and $x^2+1$ is a factor.  Thus $p(x)$ necessarily has the form $p(x) = c(x+1)(x^2+1)$ where $c$ is a positive integer.  More generally, $p(\ii ) = -b_3\ii - b_2 + b_1 \ii + b_0$ determines the end point of the walk. 
\end{proof}

We can be explicit about the coefficients of the defining polynomial as follows. 

\begin{theorem} \label{thm:poly-n4}
For a $2$-step walk of turning angle $\pi /2$ of even length $L$ from the origin to $k + \ii \ell$ ($k+\ell$ even), the defining polynomial is given by 
\begin{equation} \label{2step-poly}
p(x) =  (x+1)(x-1) \left( - \frac{\ell}{2} x - \frac{k}{2}\right) + \frac{L}{4}(x+1)(x^2+1)\,.
\end{equation} 
In particular it is uniquely defined by its length and its end point and as a consequence the (unordered) steps used to complete the walk are also uniquely defined
\end{theorem}

\begin{proof} Let $p(x) = b_3x^3+b_2x^2+b_1x+b_0$ be the defining polynomial of the walk.  The length of the walk is given by $L = p(1) = b_3+b_2+b_1+b_0$ and the end point  by $p(\ii ) = k+ \ii \ell$, so that $k = b_0 - b_2$ and $\ell = b_1 - b_3$. Then together with $-b_3+b_2-b_1 + b_0 = 0$, we can solve for $b_0,b_1,b_2,b_3$ to obtain \eqref{2step-poly}
\end{proof}

For a $2$-step walk of odd length from the origin to the point $k + \ii \ell$, we can calculate the defining polynomial $\wt{p}(x)$ of the walk to the preceeding step (a walk of even length) as above. This could be one of the four possibilities: $(k-1, \ell ), (k+1, \ell ), (k, \ell - 1), (k, \ell + 1)$.  Then we obtain $p(x)$ by adding on to $\wt{p}(x)$, $1, x^2, x, x^3$, respectively. However, now the polynomial $p(x)$ depends upon the path.  

\begin{lemma} For a given even length $L$, the number of $2$-step paths from the origin to $k + \ii \ell$ is given by
$$
\frac{2(b_2+b_0)!^2}{b_0!b_1!b_2!b_3!} = \frac{2(b_3+b_1)!^2}{b_0!b_1!b_2!b_3!}
$$
where $b_0,b_1,b_2,b_3$ are the coefficients of $p(x) = b_3x^3+b_2x^2+b_1x+b_0$ given by {\rm \eqref{2step-poly}}. The number of standard paths of length $L/2$ from the origin to $k + \ii \ell$ is half of this number. 
\end{lemma} 

\begin{proof} Horizontal steps correspond to $\pm 1$ and vertical steps to $\pm \ii$.  These must occur alternately in the walk, i.e. we must either have horizontal - vertical - horizontal - ... , or vertical - horizontal - vertical - ... Otherwise, there is no restriction on the order in which we place $+1$ and $-1$, similary for $+\ii$ and $- \ii$.  Thus the number of paths corresponds to the number of different orderings of $\{ \underbrace{1,1,\ldots 1}_{b_0}, \underbrace{-1,-1, \ldots , -1}_{b_2}\}$ multiplied by the number of different orderings of $\{ \underbrace{\ii ,\ii ,\ldots \ii}_{b_1}, \underbrace{-\ii ,-\ii , \ldots , -\ii }_{b_3}\}$. But the number of different orderings of $\{ \underbrace{1,1,\ldots 1}_{b_0}, \underbrace{-1,-1, \ldots , -1}_{b_2}\}$ is given by the binomial coefficient
$$
\left( \begin{array}{c} b_2+b_0 \\ b_0  \end{array} \right) = \left( \begin{array}{c} b_2+b_0 \\ b_2 \end{array} \right)
$$
Similarly for the vertical steps. Finally, we can begin the walk with either a horizontal step or a vertical step, so the total number of paths is given by
$$
2 \left( \begin{array}{c} b_2+b_0 \\ b_0  \end{array} \right) \left( \begin{array}{c} b_3+b_1 \\ b_1  \end{array} \right)
$$
as required. This can be written differently using the identity $b_3+b_1 = b_2+b_0$.  
\end{proof}
If in the $2$-step walk, the first step takes place with probability $1/4$ and succesive steps with probability $1/2$, then each walk of length $L$ occurs with probability
$$
\tfrac{1}{4}2^{-L+1}\,.
$$
Thus, for a given defining polynomial $p(x) = b_3x^3+b_2x^2+b_1x+b_0$, the probability that at least one walk defined by the polynomial occurs is given by 
$$
2 \left( \begin{array}{c} b_2+b_0 \\ b_0  \end{array} \right) \left( \begin{array}{c} b_3+b_1 \\ b_1  \end{array} \right) \times \tfrac{1}{4}2^{-L+1} = 4^{-L/2}  \left( \begin{array}{c} b_2+b_0 \\ b_0  \end{array} \right) \left( \begin{array}{c} b_3+b_1 \\ b_1  \end{array} \right)\,.
$$
Recalling that $L$ is even, the latter expression gives the probability that one of the corresponding standard walks occurs. 

\begin{example} There are eight $2$-step walks of length $4$ beginning and ending at the origin, each occuring with probability $2^{-3}/4$. Thus, to the corresponding polynomial $p(x) = x^3+x^2+x+1$, we associate the probability $8\times 2^{-3}/4 = 1/4$. On the other hand, there are four standard walks of length $2$ beginning and ending at the origin, each with probability $4^{-2}$, thus the probability that one of these length-$2$ walks occurs is $4 \times 4^{-2} = 1/4$. 

There are two $2$-step walks of length $2$ from the origin to the point $1+\ii$, each with associated probability $1/4 \times 1/2 = 1/8$ and so we associate the probability value $1/4$ to the polynomial $p(x) = x+1$. Equally, there is just one standard walk of length $1$ from the origin to $1+\ii$ with probability $1/4$ (recall, we identify the lattice of the standard walk with points $k+ \ii \ell$ with $k+ \ell$ even).  
\end{example}

We now explore a duality between hexagonal walks and triangular walks given by algebraic complex QCS with exterior angle $2\pi /6 = \pi /3$.  It turns out that a $2$-step walk on one of two hexagonal lattices corresponds to a $1$-step walk on the triangular lattice. Such lattices occur in the theory of random walks \cite{LL}. 

In the illustration below, the first step could be in one of three directions $1,y^2, y^4$ (grey), or one of three directions $y, y^3, y^5$ (green), where $y = e^{\ii \pi /3}$. In either case, the walk takes place entirely on either the grey hexagonal lattice or the green hexagonal lattice. After an even number of steps, the $2$-step walk will arrive at one of the points $\tfrac{3}{2}k + \ii \tfrac{\sqrt{3}}{2} \ell$ for integers $k,\ell$ with $k + \ell$ even, that is, at a point on the triangular lattice (red).  
  
\bigskip

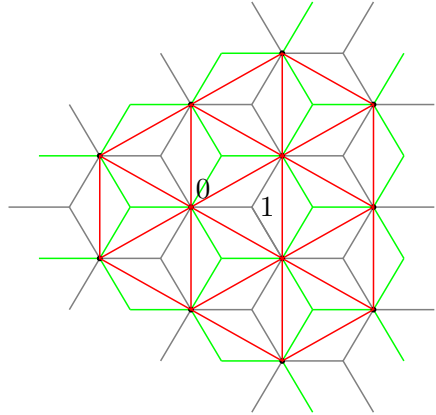
\begin{center}
\begin{tikzpicture}[line width=0.2mm,black,scale=0.4]
\draw[gray] (0,0) -- (2,0) -- (3,1.7) -- (2,3.4) -- (0,3.4) -- (-1,1.7) -- cycle;
\draw[green] (0,0) -- (1,1.7);
\draw[green] (1, 1.7) -- (3,1.7);
\draw[green] (1,1.7) -- (0,3.4);
\draw[green] (0,0) -- (-2,0);
\draw[green] (-2,0) -- (-3,1.7);
\draw[green] (-2,0) -- (-3,-1.7);
\draw[green] (0,0) -- (1,-1.7);
\draw[green] (1, -1.7) -- (3,-1.7);
\draw[green] (1,-1.7) -- (0,-3.4);
\draw[green] (3,1.7) -- (4,3.4);
\draw[green] (4,3.4) -- (6,3.4);
\draw[green] (4,3.4) -- (3,5.1);
\draw[green] (3,-1.7) -- (4,-3.4);
\draw[green] (4,-3.4) -- (6,-3.4);
\draw[green] (4,-3.4) -- (3,-5.1);
\draw[green] (3,-1.7) -- (4,0);
\draw[green] (4,0) -- (6,0);
\draw[green] (4,0) -- (3,1.7);
\draw[green] (6,3.4) -- (7,5.1);
\draw[green] (6,3.4) -- (7,1.7);
\draw[green] (6,0) -- (7,1.7);
\draw[green] (6,-3.4) -- (7,-5.1);
\draw[green] (6,-3.4) -- (7,-1.7);
\draw[green] (6,0) -- (7,-1.7);
\draw[green] (3,5.1) -- (4,6.8);
\draw[green] (3,5.1) -- (1,5.1);
\draw[green] (0,3.4) -- (1,5.1);
\draw[green] (0,3.4) -- (-2,3.4);
\draw[green] (-3,1.7) -- (-2,3.4);
\draw[green] (-3,1.7) -- (-5,1.7);
\draw[green] (3,-5.1) -- (4,-6.8);
\draw[green] (3,-5.1) -- (1,-5.1);
\draw[green] (0,-3.4) -- (1,-5.1);
\draw[green] (0,-3.4) -- (-2,-3.4);
\draw[green] (-3,-1.7) -- (-2,-3.4);
\draw[green] (-3,-1.7) -- (-5,-1.7);
\draw[gray] (2,0) -- (3,-1.7) -- (5,-1.7) -- (6,0) -- (5,1.7) -- (3,1.7);
\draw[gray] (5,1.7) -- (6,3.4) -- (5,5.1) -- (3,5.1) -- (2,3.4);
\draw[gray] (0,0) -- (-1,-1.7) -- (0,-3.4) -- (2,-3.4) -- (3,-1.7) -- (2,0);
\draw[gray] (2,-3.4) -- (3,-5.1) -- (5,-5.1) -- (6,-3.4) -- (5,-1.7);
\draw[gray] (-1,1.7) -- (-3,1.7) -- (-4,0) -- (-3,-1.7) -- (-1,-1.7);
\draw[gray] (6,-3.4) -- (8,-3.4);
\draw[gray] (6,0) -- (8,0);
\draw[gray] (6,3.4) -- (8,3.4);
\draw[gray] (5,-5.1) -- (6,-6.8);
\draw[gray] (3,-5.1) -- (2,-6.8);
\draw[gray] (0,-3.4) -- (-1,-5.1);
\draw[gray] (-3,-1.7) -- (-4,-3.4);
\draw[gray] (5,5.1) -- (6,6.8);
\draw[gray] (3,5.1) -- (2,6.8);
\draw[gray] (0,3.4) -- (-1,5.1);
\draw[gray] (-3,1.7) -- (-4,3.4);
\draw[gray] (-4,0) -- (-6,0);
\filldraw [black] (3,1.7) circle (2pt);
\filldraw [black] (3,-1.7) circle (2pt);
\filldraw [black] (0,0) circle (2pt);
\filldraw [black] (0,-3.4) circle (2pt);
\filldraw [black] (3,-5.1) circle (2pt);
\filldraw [black] (6,-3.4) circle (2pt);
\filldraw [black] (6,0) circle (2pt);
\filldraw [black] (6,3.4) circle (2pt);
\filldraw [black] (3,5.1) circle (2pt);
\filldraw [black] (0,3.4) circle (2pt);
\filldraw [black] (-3,1.7) circle (2pt);
\filldraw [black] (-3,-1.7) circle (2pt);
\node at (0.4,0.6){$0$};
\node at (2.5,0) {$1$};
\draw[red] (0,0) -- (3,1.7) -- (0,3.4) -- cycle;
\draw[red] (0,0) -- (3,-1.7) -- (3,1.7);
\draw[red] (0,0) -- (0,-3.4) -- (3,-1.7);
\draw[red] (0,-3.4) -- (3,-5.1) -- (3,-1.7);
\draw[red] (3,-5.1) -- (6,-3.4) -- (3,-1.7);
\draw[red] (6,-3.4) -- (6,0) -- (3,-1.7);
\draw[red] (3,1.7) -- (6,0) -- (6,3.4) -- (3,1.7);
\draw[red] (6,3.4) -- (3,5.1) -- (3,1.7);
\draw[red] (3,5.1) -- (-3,1.7);
\draw[red] (-3,1.7) -- (0,0) -- (-3,-1.7) -- cycle;
\draw[red] (-3,-1.7) -- (0,-3.4);
\end{tikzpicture}  
\end{center}
\bigskip

If the initial step is taken with probability $1/6$, then after two steps, the walk will arrive at one of $\tfrac{3}{2} + \tfrac{\sqrt{3}}{2}\ii , \sqrt{3}\, \ii , -\tfrac{3}{2} + \tfrac{\sqrt{3}}{2}\ii , -\tfrac{3}{2} - \tfrac{\sqrt{3}}{2}\ii , - \sqrt{3}\, \ii , \tfrac{3}{2} - \tfrac{\sqrt{3}}{2}\ii$ also with probability $1/6$, since there are two $2$-step routes to arrive at each of the points on the red lattice each with probability $1/12$. This probability distribution agrees with the case of the standard walk on the triangular lattice. 

However, unlike the case when $n=4$, the triangular walk is no longer a Markovian process, that is, one which depends only on its present position and not on past positions. Furthermore, the relation between a walk on the triangular lattice and a $2$-step walk with turning angle $\pi /3$ is more complicated. Let us examine this more closely.

\bigskip

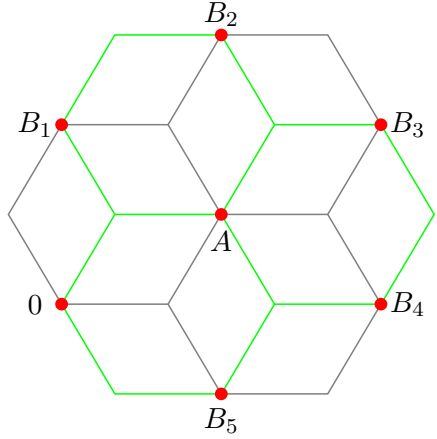
\begin{center}
\begin{tikzpicture}[line width=0.2mm,black,scale=0.7]
\draw[gray] (0,0) -- (2,0) -- (3,1.7) -- (2,3.4) -- (0,3.4) -- (-1,1.7) -- cycle;
\draw[green] (0,0) -- (1,1.7);
\draw[green] (1, 1.7) -- (3,1.7);
\draw[green] (1,1.7) -- (0,3.4);
\draw[green] (3,1.7) -- (4,3.4);
\draw[green] (0,0) -- (1,-1.7) -- (3,-1.7) -- (4,0) -- (3,1.7);
\draw[gray] (2,0) -- (3,-1.7) -- (5,-1.7) -- (6,0) -- (5,1.7) -- (3,1.7);
\node at (-0.5, 0) {$0$};
\draw[green] (4,3.4) -- (3,5.1) -- (1,5.1) -- (0,3.4);
\draw[gray] (2,3.4) -- (3,5.1) -- (5,5.1) -- (6,3.4) -- (5,1.7);
\draw[green] (4,3.4) -- (6,3.4) -- (7,1.7) -- (6,0) -- (4,0);
\filldraw [red] (0,0) circle (3pt);
\filldraw [red] (3,1.7) circle (3pt);
\filldraw [red] (0,3.4) circle (3pt);
\filldraw [red] (3,5.1) circle (3pt);
\filldraw [red] (6,3.4) circle (3pt);
\filldraw [red] (6,0) circle (3pt);
\filldraw [red] (3,-1.7) circle (3pt);
\node at (3,1.2) {$A$};
\node at (-0.5,3.4) {$B_1$};
\node at (3,5.5) {$B_2$};
\node at (6.5,3.4) {$B_3$};
\node at (6.5,0) {$B_4$};
\node at (3,-2.2) {$B_5$};
\end{tikzpicture}

\end{center}

\bigskip

If we arrive at $A$ from the origin $0$ via the grey route following successive increments $y^0$ and $y$, then there is a unique route to arrive at $B_1,B_2,B_3,B_4$ each with probability $1/4$. The nodes $B_5$ and $0$ are inaccessible. On the other hand, if we arrive at $A$ via the green route following successive increments $y$ and $y^0$, then there is a unique route to arrive at $B_2,B_3,B_4,B_5$ each with probability $1/4$. Now the nodes at $B_1$ and $0$ are inaccessible. Thus, if we are engaged in a $2$-step walk with turning angle $\pi /3$, the walk on the triangular lattice depends on the previous step and the route taken to arrive. 

Given a $1$-step walk on the triangular lattice which has arisen from a $2$-step walk with turning angle $\pi /3$, depending on the walk, there may be either a unique corresponding $2$-step walk or two such walks.  For example, the walk $0AB_1$ is determined by a unique $2$-step walk, whereas $0AB_2$ is determined by two: one on the gray lattice, one on the green lattice. In particular, if three successive steps on the triangular lattice lie on either a green, or grey hexagon, then the entire walk must have taken place on either the green hexagonal lattice, or the grey hexagonal lattice, respectively. 

As for the case when $n=4$, we can characterize the defining polynomials of $2$-step walks with turning angle $\pi /3$.

\begin{theorem} \label{thm:poly-n6}
Consider a $2$-step walk with turning angle $\pi /3$ which begins at the origin and ends at the point $\tfrac{3}{2}k + \ii \tfrac{\sqrt{3}}{2} \ell$ with $k + \ell$ even. Let $L$ be the length of the walk. Then the defining polynomial is given by
\begin{equation} \label{poly6}
p(x) = (x^3+1)(ax^2+bx-a-b+\tfrac{L}{2}) - (x-1)(x+1)\left\{ \tfrac{k}{2}(x+1) + \tfrac{\ell}{2}(x-1)\right\}
\end{equation}
where $a$ and $b$ are non-negative integers for which if $(k,\ell ) \neq (0,0)$, the following inequalities are necessary (but not sufficient) conditions: 
$$
\begin{array}{rcl}
L & \geq & \max\{ |k| + |\ell |, 2(a+b) + (k+ \ell ), 2(a+b)- (k - \ell )\} \\
a & \geq & \max \{ 0, \tfrac{k-\ell}{2} \} \\
b & \geq & \max \{ 0, - \tfrac{(k + \ell )}{2} \} 
\end{array}
$$
and if $(k,\ell ) = (0,0)$, the inequalities $a,b > 0$ and $4 \max \{ a,b\} < L < 4(a+b)$ are necessary and sufficient.  
\end{theorem}


\begin{proof} Let $p(x) = b_5x^5+b_4x^4+b_3x^3+b_2x^2+b_1x+b_0$ be the defining polynomial of the walk. Then since each increment $y^s$ is followed by either $y^{s+1}$ or $y^{s-1}$, we must have $p(-1) = 0$; also $p(1) = L$ and $p(y) = \tfrac{3}{2}k + \ii \tfrac{\sqrt{3}}{2} \ell$ where $y = e^{\ii \pi /3}$.  This yields the underdetermined system of equations
$$
\begin{array}{rcl}
b_0+b_1+b_2+b_3+b_4+b_5 & = & L \\
b_0-b_1+b_2-b_3+b_4-b_5 & = & 0 \\
2 b_0 + b_1 - b_2 - 2b_3 - b_4 + b_5 & = & 3k \\
b_1+b_2-b_4-b_5 & = & \ell 
\end{array}
$$
Set $b_5 = a$ and $b_4 = b$ as arbitrary parameters. Then on solving the system, the defining polynomial is given by
\begin{equation} \label{poly6-2}
p(x) = ax^5+bx^4+ \left( \frac{L}{2} - a - b - \frac{(k+ \ell )}{2}\right)x^3 + \left(a- \frac{(k - \ell )}{2}\right) x^2 + \left( b + \frac{(k + \ell )}{2} \right) x + \frac{L}{2} - a - b + \frac{(k - \ell )}{2}\,, 
\end{equation}
which yields \eqref{poly6}. 

Clearly the length $L$ must be greater than or equal to the minimum length of a path joining the origin to $\tfrac{3}{2}k + \ii \tfrac{\sqrt{3}}{2} \ell$ which is given by $|k|+|\ell |$.  In the case when $(k, \ell ) \neq (0,0)$, the other inequalities arise from the requirement that the coefficients of $p(x)$ must all be non-negative. If on the other hand $k = \ell = 0$ and the path is closed, then the condition on $L$ is determined by the case $n = 6$ in \S\ref{sec:complex}. 
Recall, writing the polynomial $p(x) = (x^3+1)(ax^2+bx+c)$, we require $a,b,c >0$ and $\max \{ a,b,c\} <  a+b+c - \max \{ a,b,c\}$, where $L = 2(a+b+c)$ in the notation of \S\ref{sec:complex}. 

Suppose $\frac{L}{2}  - a - b \geq \max\{ a,b\}$, i.e. $L \geq 2(a+b) + 2 \max \{ a,b\}$.  Then we require
$$
\frac{L}{2} - a - b < a+b \quad \Leftrightarrow \quad  L < 4(a+b).
$$
On the other hand, if $\frac{L}{2} - a - b \leq \max \{ a,b\}$, i.e. $L \leq 2(a+b) + 2 \max \{ a,b\}$, then we require
$$
\max \{ a,b\} < \frac{L}{2} - a - b + \min \{ a,b\} \quad \Leftrightarrow \quad 4\max \{ a,b\} < L\,.
$$
However, for $a,b > 0$, we have
$$
\left. \begin{array}{cc}
either & 2(a+b)+2\max \{ a,b\} \leq L < 4(a+b) \\
or & 4\max \{ a,b\} < L \leq 2(a+b) + 2\max \{ a,b\} \end{array} \right\} \quad \Leftrightarrow \quad 4\max \{ a,b\} < L < 4(a+b)
$$
as required. 
\end{proof}  

The inequalities for the case $(k,\ell ) \neq (0,0)$ in the above theorem are necessary conditions but not sufficient. For example, if we choose $k=3$, $\ell = 5$, $L=10$, $a=0$, $b=1$, then the inequalities are satisfied, however, the resulting polynomial given by $p(x) = x^4+x^2+5x+3$ is not legitimate.

Note that for the case $n=6$, the defining polynomial may depend on the path.   As the following example shows, this allows one to construct paths of the same length to the same point which make use of different (unordered) edges.

\begin{example} \label{ex:different-paths} In \eqref{poly6-2}, set $k=3$, $\ell = 5$ and $L=12$. First choose $a=1$ and $b=0$ to give the defining polynomial 
$$
p(x) = x^5+x^3+2x^2+4x+4
$$
with sequence of increments
$$
(y^5,1,y,1,y,1,y,1,y,y^2,y^3,y^2)
$$
corresponding to the red path in the illustration below.

\bigskip

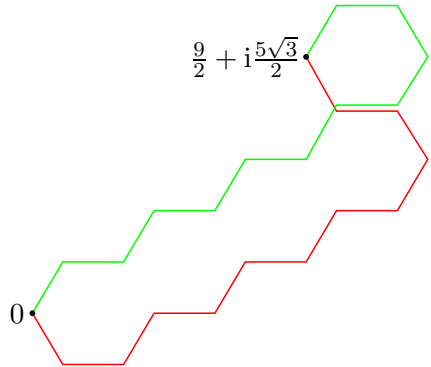
\begin{center}
\begin{tikzpicture}[line width=0.2mm,black,scale=0.4]
\draw[red] (0,0) -- (1,-1.7) -- ( 3,-1.7) -- (4,0) -- (6,0) -- (7,1.7) -- (9,1.7) -- (10,3.4) -- (12,3.4) -- (13,5.1) -- (12,6.7) -- (10,6.7) -- (9,8.5); 
\draw[green] (0,0) -- (1,1.7) -- (3,1.7) -- (4,3.4) -- (6,3.4) -- (7,5.1) -- (9,5.1) -- (10,6.9) -- (12,6.9) -- (13,8.5) -- (12,10.2) -- (10,10.2) -- (9,8.5);
\filldraw [black] (0,0) circle (2pt);
\filldraw [black] (9,8.5) circle (2pt);
\node at (-0.5,0) {$0$};
\node at (7,8.5) {$\tfrac{9}{2} + \ii \tfrac{5\sqrt{3}}{2}$};
\end{tikzpicture}  
\end{center}
\bigskip
Second, choose $a=0$ and $b=1$ to give the defining polynomial
$$
p(x) = x^4+x^3+x^2+5x+4
$$
with sequence of increments 
$$
(y,1,y,1,y,1,y,1,y,y^2,y^3,y^4)
$$
corresponding to the green path in the illustration. Then the red path uses a different set of edges to the green path, for example it exploits the edge $y^5$ which is not used in the green path. Note that the two paths combine to yield a closed path of length 24. As affirmed by Proposition \ref{prop:symmetry}, this closed path requires each edge with its oppositely oriented counterpart. 
\end{example}  

We can proceed similarly with turning angle $2 \pi /n$ for $n=8,10,\ldots $, however in general, there are no longer convenient tilings of the plane which support the walks. We illustrate below part of the lattice for the case $n=8$, where we see how the brown octagons begin to interfere with the tiling.  Some points of the $2$-step walk are illustrated as red nodes.

\bigskip

\begin{center}
\begin{tikzpicture}[line width=0.2mm,black,scale=0.7]
\draw[gray] (0,0) -- (1,0) -- (1.7,0.7) -- (1.7,1.7) -- (1,2.4) -- (0,2.4) -- (-0.7, 1.7) -- (-0.7,0.7) --cycle;
\draw[green] (0,0) -- (0.7,0.7) -- (0.7,1.7) -- (0,2.4) -- (-1,2.4) -- (-1.7,1.7) -- (-1.7,0.7) -- (-1,0) -- cycle;
\draw[blue] (0,0) -- (0,1) -- (-0.7,1.7) -- (-1.7,1.7) -- (-2.4,1) -- (-2.4,0) -- (-1.7,-0.7) -- (-0.7,-0.7) -- cycle;
\draw[gray] (-0.7,0.7) -- (-1.7,0.7) -- (-2.4,0) -- (-2.4,-1) -- (-1.7,-1.7) -- (-0.7,-1.7) -- (0,-1) -- (0,0);
\draw[green] (-1,0) -- (-1.7,-0.7) -- (-1.7,-1.7) -- (-1,-2.4) -- (0,-2.4) -- (0.7,-1.7) -- (0.7,-0.7) -- (0,0);
\draw[blue] (-0.7,-0.7) -- (-0.7,-1.7) -- (0,-2.4) -- (1,-2.4) -- (1.7,-1.7) -- (1.7,-0.7) -- (1,0);
\draw[gray] (0,-1) -- (0.7,-1.7) -- (1.7,-1.7) -- (2.4,-1) -- (2.4,0) -- (1.7,0.7) -- (0.7,0.7);
\draw[green] (0.7,-0.7) -- (1.7,-0.7) -- (2.4,0) -- (2.4,1) -- (1.7,1.7) -- (0.7,1.7) -- (0,1);
\draw[brown] (1.7,0.7) -- (2.7,0.7) -- (3.4,0) -- (3.4,-1) -- (2.7,-1.7) -- (1.7,-1.7) -- (1,-1) -- (1,0);
\draw[brown] (1,-1) -- (0,-1) -- (-0.7,-0.3) -- (-0.7,0.7) -- (0,1.4) -- (1,1.4) -- (1.7,0.7) -- (1.7,-0.3) -- (1,-1);
\filldraw [red] (0,0) circle (2pt);
\filldraw [red] (1.7,0.7) circle (2pt);
\filldraw [red] (1.7,-0.7) circle (2pt);
\filldraw [red] (-0.7,1.7) circle (2pt);
\filldraw [red] (0.7,1.7) circle (2pt);
\filldraw [red] (-1.7,0.7) circle (2pt);
\filldraw [red] (-1.7,-0.7) circle (2pt);
\filldraw [red] (-0.7,-1.7) circle (2pt);
\filldraw [red] (0.7,-1.7) circle (2pt);
\filldraw [red] (3.4,0) circle (2pt);
\filldraw [red] (2.7,-1.7) circle (2pt);
\filldraw [red] (1,-1) circle (2pt);
\filldraw [red] (-0.7,-0.3) circle (2pt);
\filldraw [red] (0,1.4) circle (2pt);
\end{tikzpicture}

\end{center}

\end{document}